\documentclass[a4paper]{amsart}
\pdfoutput=1
\usepackage[utf8]{inputenc}
\usepackage[T1]{fontenc}
\usepackage{lmodern}
\usepackage{amssymb}
\usepackage{nicefrac,mathtools,enumitem,amsmath,amsxtra}
\usepackage{microtype}
\usepackage{tikz-cd}

\usepackage[pdftitle={Geometric construction of classes in van Daele's K-theory},
pdfauthor={Collin Mark Joseph and Ralf Meyer},
pdfsubject={Mathematics}
]{hyperref}
\usepackage[lite]{amsrefs}

\renewcommand{\PrintDOI}[1]{\href{http://dx.doi.org/\detokenize{#1}}{doi: \detokenize{#1}}}

\BibSpec{book}{%
  +{}  {\PrintPrimary}                {transition}
  +{,} { \textit}                     {title}
  +{.} { }                            {part}
  +{:} { \textit}                     {subtitle}
  +{,} { \PrintEdition}               {edition}
  +{}  { \PrintEditorsB}              {editor}
  +{,} { \PrintTranslatorsC}          {translator}
  +{,} { \PrintContributions}         {contribution}
  +{,} { }                            {series}
  +{,} { \voltext}                    {volume}
  +{,} { }                            {publisher}
  +{,} { }                            {organization}
  +{,} { }                            {address}
  +{,} { \PrintDateB}                 {date}
  +{,} { }                            {status}
  +{}  { \parenthesize}               {language}
  +{}  { \PrintTranslation}           {translation}
  +{;} { \PrintReprint}               {reprint}
  +{.} { }                            {note}
  +{.} {}                             {transition}
  +{} { \PrintDOI}                   {doi}
  +{} { available at \url}            {eprint}
  +{}  {\SentenceSpace \PrintReviews} {review}
}

\setlist[enumerate,1]{label=\textup{(\arabic*)}}
\setlist[enumerate,2]{label=\textup{(\alph*)}}

\theoremstyle{plain}
\newtheorem{theorem}[subsection]{Theorem}
\newtheorem{lemma}[subsection]{Lemma}
\newtheorem{proposition}[subsection]{Proposition}

\theoremstyle{definition}
\newtheorem{definition}[subsection]{Definition}

\theoremstyle{remark}
\newtheorem{remark}[subsection]{Remark}
\newtheorem{example}[subsection]{Example}

\newcommand*{\defeq}{\mathrel{\vcentcolon=}}
\newcommand*{\congto}{\xrightarrow\sim}
\newcommand*{\alb}{\hspace{0pt}} 

\DeclarePairedDelimiter{\abs}{\lvert}{\rvert}
\DeclarePairedDelimiter{\norm}{\lVert}{\rVert}
\DeclarePairedDelimiterX{\setgiven}[2]{\{}{\}}{#1\,{:}\,\mathopen{}#2}

\newcommand*{\into}{\rightarrowtail}
\newcommand*{\prto}{\twoheadrightarrow}

\newcommand*{\Mult}{\mathcal M}
\newcommand*{\U}{\mathcal U}
\newcommand{\C}{\mathbb C}
\newcommand{\N}{\mathbb N}
\newcommand{\Z}{\mathbb Z}
\newcommand{\R}{\mathbb R}
\newcommand{\Qut}{\mathbb H}
\newcommand{\T}{\mathbb T}
\newcommand{\Sphere}{\mathbb S}

\newcommand{\hot}{\mathbin{\hat\otimes}}

\newcommand*{\Cont}{\mathrm C}
\newcommand*{\Contb}{\mathrm{C_b}}
\newcommand{\rinv}{\mathfrak{r}}

\newcommand*{\pt}{\mathrm{pt}}
\newcommand*{\Star}{$^*$\nobreakdash-}
\newcommand*{\nb}{\nobreakdash}
\newcommand*{\Cst}{\mathrm C^*}

\newcommand*{\ima}{\mathrm i}
\newcommand{\Comp}{\mathbb K}
\newcommand{\Bound}{\mathbb B}
\newcommand{\Mat}{\mathbb M}
\newcommand{\Id}{\mathrm{id}}
\newcommand{\ev}{\mathrm{ev}}
\newcommand{\an}{\mathrm{an}}

\newcommand*{\conj}[1]{\overline{#1}}
\newcommand*{\cl}[1]{\overline{#1}}

\newcommand{\FU}{\mathrm{FU}}
\newcommand{\GFU}{\mathrm{GFU}}

\newcommand{\K}{\mathrm{K}}
\newcommand{\DK}{\mathrm{DK}}
\newcommand{\KO}{\mathrm{KO}}
\newcommand{\KR}{\mathrm{KR}}
\newcommand{\KK}{\mathrm{KK}}
\newcommand{\KKR}{\mathrm{KKR}}

\DeclareMathOperator{\Hom}{Hom}
\DeclareMathOperator{\Ker}{Ker}
\DeclareMathOperator{\Cliff}{Cl}
\DeclareMathOperator{\sign}{sign}

\begin{document}
\title[Geometric construction of classes in van Daele's K-theory]{Geometric construction of\\ classes in van Daele's K-theory}
\author{Collin Mark Joseph}
\email{collinmark.joseph@stud.uni-goettingen.de}
\author{Ralf Meyer}
\email{rmeyer2@uni-goettingen.de}

\address{Mathematisches Institut\\
  Universit\"at G\"ottingen\\Bunsenstra\ss e 3--5\\
  37073 G\"ottingen\\Germany}



\begin{abstract}
  We describe explicit generators for the ``real'' K-theory of
  ``real'' spheres in van Daele's picture.  Pulling these generators
  back along suitable maps from tori to spheres produces a family of
  Hamiltonians used in the physics literature on topological
  insulators.  We compute their K-theory classes geometrically,
  based on wrong-way functoriality of K-theory and the geometric
  version of bivariant K-theory, which we extend to the ``real''
  case.
\end{abstract}
\maketitle

\section{Introduction}
\label{sec:intro}

Topological insulators are insulators that, nevertheless, conduct
electricity on their boundaries.  Even more, conducting states on
the boundary are forced to be present by topological obstructions.
This suggests that the boundary conducting states are quite robust
under disorder.  In the one-particle approximation, topological
insulators may be classified by the topological \(\K\)\nb-theory of
the observable \(\Cst\)\nb-algebra.  In translation-invariant
tight-binding models, the observable algebra is isomorphic to a
matrix algebra over the algebra of continuous functions on the
\(d\)-torus~\(\T^d\), where~\(d\) is the dimension of the material.
Many interesting phenomena arise when the Hamiltonian enjoys extra
symmetries that are anti-unitary or anticommute with it.
Altogether, there are ten different symmetry types, and these
correspond to the two complex and the eight real \(\K\)\nb-theory
groups.  More precisely, the torus appears through the Fourier
transform, and the relevant observable algebra is the group
\(\Cst\)\nb-algebra of~\(\Z^d\) with real coefficients.  Under
Fourier transform, this becomes the real \(\Cst\)\nb-algebra
\[
  \setgiven*{f\colon \T^d \to \C}{f(\conj{z}) = \conj{f(z)} \text{
      for all }z\in\T^d}.
\]
The conjugation maps used here is an involution and makes the torus
a ``real'' space.  This is the same as a space with an action of the
group~\(\Z/2\).  In the following, we denote ``real'' structures
as~\(\rinv\).  For the \(d\)\nb-torus, we get
\begin{equation}
  \label{eq:Torus_d_detail}
  \begin{gathered}
    \T^d = \setgiven{(x_1,\dotsc,x_d,y_1,\dotsc,y_d)\in\R^{2 d}}{x_j^2
      + y_j^2 = 1\text{ for }j=1,\dotsc,d},\\
    \rinv_{\T^d}(x_1,\dotsc,x_d,y_1,\dotsc,y_d) \defeq
    (x_1,\dotsc,x_d,-y_1,\dotsc,-y_d).
  \end{gathered}
\end{equation}
As a result, the relevant \(\K\)\nb-theory for the study of
topological phases is the ``real'' \(\K\)\nb-theory \(\KR^*(\T^d)\)
of \((\T^d,\rinv_{\T^d})\) as defined by
Atiyah~\cite{Atiyah:K_Reality}.  This appearance in physics has
renewed the interest in ``real'' \(\K\)\nb-theory.

This article generalises geometric bivariant \(\K\)\nb-theory as a tool for
\(\K\)\nb-theory computations to the ``real'' case and uses this to compute
the \(\K\)\nb-theory classes of certain Hamiltonians studied in the complex
case already in~\cite{Prodan-Schulz-Baldes:Bulk_boundary}.  In
addition, we describe explicit generators for the ``real'' \(\K\)\nb-theory
of spheres, extending a formula by Karoubi
in~\cite{Karoubi:Lectures_K} for the complex \(\K\)\nb-theory of
even-dimensional spheres.

There have always been several different ways to describe the
\(\K\)\nb-theory of a space or a \(\Cst\)\nb-algebra.  As noted by
Kellendonk~\cite{Kellendonk:Cstar_phases}, the \(\K\)\nb-theory picture that
is closest to the classification of topological insulators is van
Daele's picture.  His definition applies to a real or complex
\(\Cst\)\nb-algebra~\(A\) with a \(\Z/2\)-grading.  It is based on
odd, selfadjoint unitaries in matrix algebras over~\(A\).  A
\(\Z/2\)-grading may be interpreted physically as a chiral symmetry,
and an odd selfadjoint unitary is just the spectral flattening of a
Hamiltonian with a spectral gap at zero that respects the given
chiral symmetry.  Systems without chiral symmetry may also be
treated by doubling the number of degrees of freedom to introduce an
auxiliary chiral symmetry that anticommutes with the Hamiltonian.

The starting point of this article was the discussion by Prodan and
Schulz-Baldes~\cite{Prodan-Schulz-Baldes:Bulk_boundary} of certain
examples of Hamiltonians~\(H_m\) in any dimension~\(d\), namely,
\[
H_m \defeq \frac{1}{2\ima} \sum_{j=1}^d (S_j - S_j^*) \otimes \gamma_j
+ \Bigl(m + \frac{1}{2} \sum_{j=1}^d (S_j + S_j^*)\Bigr)
\otimes \gamma_0
\in\Cst(\Z^d)\otimes\Cliff_{1,d}
\]
with Clifford generators \(\gamma_0,\dotsc,\gamma_d\) and
translations~\(S_j\) in coordinate directions for \(j=1,\dotsc,d\),
and a mass parameter~\(m\) (see
\cite{Prodan-Schulz-Baldes:Bulk_boundary}*{§2.2.4 and §2.3.3}).  The
selfadjoint element~\(H_m\) has a spectral gap at zero if and only
if \(m\notin \{-d,-d+2,\dotsc,d-2,d\}\).  Then it defines an
insulator.  The top-degree Chern character of its \(\K\)\nb-theory
class and its jumps at the values in \(\{-d,-d+2,\dotsc,d-2,d\}\)
are computed in the physics literature (see
\cite{Prodan-Schulz-Baldes:Bulk_boundary}*{Equation~(2.26)} and also
\cites{Golterman-Jansen-Kaplan:Currents,
  Qi-Hughes-Zhang:TFT_time-reversal}).  Here we explain a possible
mathematical origin of~\(H_m\): it is the pullback of a generator of
the reduced \(\KR\)-theory of a sphere along a map
\(\varphi_m\colon \T^d \to \Sphere^{1,d}\).  Then we proceed to
compute the class of~\(H_m\) in ``real'' \(\K\)\nb-theory for
all~\(m\) and all dimensions~\(d\).

The first step for this is to describe explicit generators for the
\(\KR\)\nb-theory of ``real'' spheres in van Daele's picture.
Let~\(\R^{a,b}\) denote \(\R^a\times \R^b\) with the involution
\(\rinv(x,y) \defeq (x,-y)\) for \(x\in\R^a\), \(y\in\R^b\).  Let
\(\Sphere^{a,b} \subseteq \R^{a,b}\) be the unit sphere with the
restricted real involution.  Let \(\Cliff_{a,b}\) denote the
Clifford algebra with \(a+b\) anticommuting, odd, selfadjoint,
unitary generators \(\gamma_1,\dotsc,\gamma_{a+b}\), and such that
\(\gamma_1,\dotsc,\gamma_a\) and
\(\ima \gamma_{a+1},\dotsc, \ima \gamma_{a+b}\) are real.  Then
\begin{equation}
  \label{eq:beta}
  \beta_{a,b}
  \defeq \sum_{j=1}^{a+b} x_j \gamma_j \Bigr|_{\Sphere^{a,b}}
\end{equation}
is an odd, selfadjoint, unitary, and real element of the
\(\Cst\)\nb-algebra \(\Cont(\Sphere^{a,b}) \otimes \Cliff_{a,b}\).
Therefore, it defines a class in its van Daele \(\K\)\nb-theory, which is
isomorphic to the ``real'' \(\K\)\nb-theory group
\(\KR^{a-b-1}(\Sphere^{a,b})\).
We check that its image in the reduced ``real'' \(\K\)\nb-theory is a
generator in the sense that the exterior product map with it defines
an isomorphism from the ``real'' \(\K\)\nb-theory of a point to the reduced
``real'' \(\K\)\nb-theory of~\(\Sphere^{a,b}\).  Our proof
that~\(\beta_{a,b}\) generates the ``real'' \(\K\)\nb-theory is based
on the proof of Bott periodicity by
Kasparov~\cite{Kasparov:Operator_K} and Roe's proof of the
isomorphism between \(\KK_0(\R,A)\) and van Daele's \(\K\)\nb-theory
of~\(A\) for any \(\Z/2\)-graded real \(\Cst\)\nb-algebra~\(A\).

The spectral flattening of the Hamiltonian~\(H_m\) for
\(m\in\R\setminus \{-d,-d+2,\dotsc,d-2,d\}\) is an odd,
selfadjoint, real unitary on the \(d\)\nb-torus~\(\T^d\) with the
involution by complex conjugation in each circle factor.  This may
be written down as the pull-back of~\(\beta_{1,d}\) along the
real map
\begin{equation}
  \label{eq:varphi_m}
  \begin{aligned}
    \varphi_m\colon \T^d &\to \Sphere^{1,d},\\
    (x_1,\dotsc,x_d,y_1,\dotsc,y_d) &\mapsto
    \frac{(x_1+\dotsb+x_d+m,y_1,y_2,\dotsc,y_d)}
    {\norm{(x_1+\dotsb+x_d+m,y_1,y_2,\dotsc,y_d)}}.
  \end{aligned}
\end{equation}
Our task is to compute this pull back in the \(\KR\)\nb-theory of the
torus.

There is another way to describe the generator of
\(\KR^{a-b-1}(\Sphere^{a,b})\) for \(a>0\).  Then the two ``poles''
\begin{equation}
  \label{eq:poles}
  N \defeq (1,0,\ldots, 0),\qquad
  S \defeq (-1,0,\ldots,0),
\end{equation}
are fixed by the ``real'' involution on~\(\Sphere^{a,b}\).  The
stereographic projection identifies
\(\Sphere^{a,b} \setminus \{N\}\) with~\(\R^{a-1,b}\) as a ``real''
manifold.  Thus Bott periodicity identifies the reduced
\(\KR\)\nb-theory of~\(\Sphere^{a,b}\) with
\(\KR^{a-b-1}(\R^{a-1,b}) \cong \KR^0(\pt) \cong \Z\).  The
resulting map \(\KR^0(\pt) \to \KR^{a-b-1}(\Sphere^{a,b})\) is an
example of the wrong-way functoriality of \(\K\)\nb-theory.  Namely,
it is the shriek map~\(i!\) associated to the inclusion map~\(i\)
of~\(\{S\}\).

Shriek maps and pull-back maps such as \(i!\) and~\(\varphi_m^*\)
are among the building blocks of geometric bivariant
\(\K\)\nb-theory.  This theory also allows to compute the composites
of such maps geometrically.  In our case, this says that
\(\varphi_m^* \circ i!\) is the sum of the shriek maps for all
points in \(\varphi_m^{-1}(S)\), equipped with appropriate
orientations.  Our main result, Theorem~\ref{the:compute}, computes
the image of this in the usual direct sum decomposition of
\(\KR^*(\T^d)\).

To make this computation valid in the ``real'' case, we show that
the geometric bivariant \(\K\)\nb-theory as developed in
\cites{Emerson-Meyer:Normal_maps, Emerson-Meyer:Correspondences}
still works for \(\KR\)\nb-theory.  These articles define geometric
bivariant \(\K\)\nb-theory in a slightly different way than
suggested originally by
Connes--Skandalis~\cite{Connes-Skandalis:Longitudinal}, in order to
extend it more easily to the equivariant case.  A ``real''
involution on a space is the same as a \(\Z/2\)\nb-action.  The
``real'' \(\K\)\nb-theory is not the same as \(\Z/2\)\nb-equivariant
\(\K\)\nb-theory.  The difference is that \(\Z/2\)\nb-equivariant
\(\K\)\nb-theory looks at \(\Z/2\)-actions on vector bundles that
are fibrewise linear, whereas ``real'' \(\K\)\nb-theory looks at
\(\Z/2\)-actions on vector bundles that are fibrewise
conjugate-linear.  This change in the setup does not affect the
properties of equivariant \(\K\)\nb-theory that are needed to
develop bivariant equivariant \(\K\)\nb-theory.  We only comment on
this rather briefly in this article.  More details may be found in
the Master's Thesis~\cite{Joseph:Master}.

\subsection{Some basic conventions}

Throughout the article, we let \(\R^{a,b}\) for \(a,b\in\N\) be
\(\R^a\times \R^b\) with the involution \(\rinv(x,y) \defeq (x,-y)\)
for \(x\in\R^a\), \(y\in\R^b\), and we let \(\Sphere^{a,b}\) be the
unit sphere in~\(\R^{a,b}\); so the dimension of~\(\Sphere^{a,b}\)
is \(a+b-1\).  Our \(\R^{a,b}\) and~\(\Sphere^{a,b}\) are denoted
\(\R^{b,a}\) and~\(\Sphere^{b,a}\) by
Atiyah~\cite{Atiyah:K_Reality}; our notation is that of
Kasparov~\cite{Kasparov:Operator_K}.

A ``real'' structure on a \(\Cst\)\nb-algebra~\(A\) is a
conjugate-linear, involutive \Star{}\alb{}homomorphism
\(\rinv\colon A\to A\).  Then
\[
  A_\R \defeq \setgiven{a\in A}{\rinv(a) = a}
\]
is a real \(\Cst\)\nb-algebra such that \(A\cong A_\R \otimes \C\)
with the involution
\(\rinv(a\otimes \lambda) \defeq a \otimes \conj{\lambda}\).  Thus
``real'' \(\Cst\)\nb-algebras are equivalent to real
\(\Cst\)\nb-algebras.  Any commutative ``real'' \(\Cst\)\nb-algebra
is isomorphic to \(\Cont_0(X)\) with the real involution
\(\rinv(f)(x) \defeq \conj{f(\rinv(x))}\) for all \(x\in X\) for a
``real'' locally compact space \((X,\rinv)\).

Let~\(\Cliff_{a,b}\) denote the complex Clifford algebra with
\(a+b\) anticommuting, odd, selfadjoint, unitary generators
\(\gamma_1,\dotsc,\gamma_{a+b}\), and such that
\(\gamma_1,\dotsc,\gamma_a\) and
\(\ima \gamma_{a+1},\dotsc, \ima \gamma_{a+b}\) are real.  This is a
\(\Z/2\)-graded ``real'' \(\Cst\)\nb-algebra.  Let~\(\hot\) denote
the graded-commutative tensor product for \(\Z/2\)-graded (``real'')
\(\Cst\)\nb-algebras.

For \(\R^{1,d}\) and its subspace~\(\Sphere^{1,d}\) and for the
corresponding Clifford algebra~\(\Cliff_{1,d}\), it is convenient to
start numbering at~\(0\).  That is, the coordinates in~\(\R^{1,d}\)
are \(x_0,\dotsc,x_d\) and~\(\Cliff_{1,d}\) is has the
anticommuting, odd, selfadjoint generators
\(\gamma_0,\dotsc,\gamma_d\), such that
\(\gamma_0, \ima \gamma_1,\dotsc, \ima \gamma_d\) are real.

\section{Explicit \texorpdfstring{$\K$\nb-}{K-}theory of spheres}
\label{sec:4}

We are going to write down explicit generators for the reduced
\(\K\)\nb-theory groups of spheres.  This depends, of course, on the
definition of \(\K\)\nb-theory that we are using.  In terms of
vector bundles given by gluing, Atiyah--Bott--Shapiro describe
in~\cite{Atiyah-Bott-Shapiro:Clifford} how to go from irreducible
Clifford modules to these generators.  They do treat
\(\KO\)\nb-theory, but not the ``real'' case, which is only invented
in~\cite{Atiyah:K_Reality}.  For the description of~\(\K_0\) through
projections, Karoubi~\cite{Karoubi:Lectures_K} has written down
explicit generators for the complex \(\K\)\nb-theory of
even-dimensional spheres.  Kasparov~\cite{Kasparov:Operator_K} has
written down explicit generators for his bivariant \(\KK\)-theory
\(\KKR_0(\C,\Cont_0(\R^{a,b}) \otimes \Cliff_{a,b})\); here we write
\(\KKR\) to highlight that this means the ``real'' version of the
theory.

We want our explicit generators in van Daele's \(\K\)\nb-theory for
\(\Z/2\)-graded ``real'' \(\Cst\)\nb-algebras; this version of
\(\K\)\nb-theory is ideal for the applications to topological
insulators that we have in mind
(see~\cite{Kellendonk:Cstar_phases}), and it is also helpful to
treat ``real'' spheres of all dimensions simultaneously.  Actually,
our generators are just those of Kasparov, translated along a
canonical isomorphism between Kasparov's KK-theory and van Daele's
\(\K\)\nb-theory.  This isomorphism is an auxiliary result due to
Roe~\cite{Roe:Paschke_real}.  Roe's isomorphism is discussed in
greater detail in~\cite{Bourne-Kellendonk-Rennie:Cayley_K-theory},
which also relates several variants of van Daele's \(\K\)\nb-theory.

\subsection{Van Daele \texorpdfstring{$\K$\nb-}{K-}theory for graded C*-algebras and Roe's
isomorphism}

The following definitions make sense both for real and complex
\(\Cst\)\nb-algebras.  We write them down in the ``real'' case
because this is what we are going to use later.  The effect of the
condition \(\rinv(a)=a\) below is to replace a ``real''
\(\Cst\)\nb-algebra~\(A\) by its real subalgebra \(A_\R \defeq
\setgiven{a\in A}{\rinv(a)=a}\).

\begin{definition}
  Let~\(A\) be a unital, \(\Z/2\)-graded ``real'' \(\Cst\)\nb-algebra
  and \(n \in\N_{\ge1}\).  Let
  \[
    \mathcal{FU}_n(A) \defeq \setgiven{a \in \Mat_n A}{a=a^*,\
      \rinv(a) = a,\ a^2=1,\ a \text{ odd}},
  \]
  equipped with the norm topology of~\(\Mat_n A\).  Two elements in
  \(\mathcal{FU}_n(A)\) in the same path components are called
  \emph{homotopic}.  We define the maps
  \[
    {\oplus}\colon \mathcal{FU}_n(A) \times \mathcal{FU}_m(A) \to
    \mathcal{FU}_{n+m}(A),\qquad
    (a,b) \mapsto
    \begin{pmatrix}
      a&0\\0&b
    \end{pmatrix}.
  \]
  Let \(\FU_n(A) \defeq \pi_0(\mathcal{FU}_n(A))\).  We abbreviate
  \(\mathcal{FU}(A) \defeq \mathcal{FU}_1(A)\) and
  \(\FU(A) \defeq \FU_1(A)\).  We call~\(A\) \emph{balanced} if
  \(\mathcal{FU}(A)\neq \emptyset\).
\end{definition}

The condition \(a=a^*\) may be dropped (see
\cite{VanDaele:K_graded_I}*{Proposition 2.5}).  For a unitary, it
means that \(a^2=1\), forcing the spectrum to be contained
in~\(\{\pm1\}\).  Then \((a+1)/2\) is a projection.  It is useless,
however, because it is not even.

The operation~\(\oplus\) is associative and commutative up to
homotopy (see \cite{VanDaele:K_graded_I}*{Proposition 2.7}).  So
\(\bigsqcup_{n=1}^\infty \FU_n(A)\) becomes an Abelian semigroup.
Let \(\GFU(A)\) be its Grothendieck group.

\begin{definition}
  Let~\(A\) be a balanced, unital, \(\Z/2\)-graded ``real''
  \(\Cst\)\nb-algebra.  The \emph{van Daele \(\K\)\nb-theory}
  \(\DK(A)\) of~\(A\) is defined as the kernel of the homomorphism
  \(d\colon \GFU(A) \to \Z\) defined by \(d|_{\FU_n(A)} = n\) for
  all \(n\in\N_{\ge1}\).
\end{definition}

Van Daele's original definition uses an element
\(e\in\mathcal{FU}(A)\) with \(e\sim (-e)\) to define stabilisation
maps
\[
  \FU_n(A) \to \FU_{n+1}(A),\qquad
  [x] \mapsto [x \oplus e].
\]
The colimit of the resulting inductive system becomes an Abelian
group under~\(\bigoplus\).  This group is isomorphic to~\(\DK(A)\)
as defined above (see \cites{Roe:Paschke_real,
  Bourne-Kellendonk-Rennie:Cayley_K-theory}).

The definition above is generalised beyond the balanced, unital case
as follows.  First let~\(A\) be a unital, \(\Z/2\)-graded ``real''
\(\Cst\)\nb-algebra.  Then \(A \hot \Cliff_{1,1}\) is balanced
because \(1\hot \gamma_1 \in \FU(A\hot \Cliff_{1,1})\).  If~\(A\) is
already balanced, then there is a natural stabilisation isomorphism
\(\DK(A) \cong \DK(A \hot \Cliff_{1,1})\).  This justifies defining
\(\DK(A) \defeq \DK(A \hot \Cliff_{1,1})\) in general.  For a
\(\Cst\)\nb-algebra~\(A\) without unit, let~\(A^+\) be its
unitisation, equipped with the canonical augmentation homomorphism
\(A^+ \to \C\).  Then
\[
  \DK(A) \defeq \Ker\bigl( \DK(A^+) \to \DK(\C)\bigr).
\]
This reproduces the previous definition if~\(A\) is unital because
then \(A^+ \cong A \oplus \C\) and \(\DK\) is additive.  Since there
is an isomorphism \(\DK(A) \cong \KKR_0(\Cliff_{1,0},A)\), the
functor~\(\DK\) is stable with respect to matrix algebras, Morita
invariant, homotopy invariant, and exact for all extensions of
\(\Z/2\)-graded \(\Cst\)\nb-algebras.  See
also~\cite{Bourne-Kellendonk-Rennie:Cayley_K-theory} for direct
proofs of these properties of van Daele's \(\K\)\nb-theory using the
definition above.

\subsection{Roe's isomorphism}

The following proposition is due to Roe~\cite{Roe:Paschke_real}.
The equivalent isomorphism \(\DK(A) \cong\KKR_0(\Cliff_{1,0},A)\) is
proven in~\cite{Bourne-Kellendonk-Rennie:Cayley_K-theory}.

\begin{proposition}[\cite{Roe:Paschke_real}]
  \label{pro:Daele_vs_KK}
  Let~\(A\) be a unital, \(\Z/2\)-graded ``real'' \(\Cst\)\nb-algebra.
  Then
  \(\DK(A \hot \Cliff_{r+1,s}) \cong \KKR_0(\C, A \hot
  \Cliff_{r,s})\).
\end{proposition}

This justifies defining the ``real'' \(\K\)\nb-theory \(\KR(A)\) and
its graded version for a \(\Z/2\)-graded ``real''
\(\Cst\)\nb-algebra~\(A\) as
\[
  \KR(A) \defeq \DK(A \hot \Cliff_{1,0}),\qquad
  \KR_n(A) \defeq \DK(A \hot \Cliff_{1,n}),
\]
for \(n\in\N\).  Actually, this only depends on \(n\bmod 8\), so
that we may also take \(n\in\Z/8\).  If~\(A\) is trivially graded,
then \(\KR(A)\) is naturally isomorphic to the ordinary~\(\K_0\) of
the real Banach algebra~\(A_\R\).  This is sometimes denoted
\(\KO_0(A_\R)\) to highlight that~\(A_\R\) is only a real
\(\Cst\)\nb-algebra.

For a ``real'' locally compact space~\(X\), we define its ``real''
\(\K\)\nb-theory \(\KR(X)\) as \(\KR(\Cont_0(X))\) with the induced
``real'' structure.  This is equivalent to the definition by
Atiyah~\cite{Atiyah:K_Reality}.  For \(n\in\Z/8\), we also let
\[
  \KR^n(X) \defeq \KR_{-n}(\Cont_0(X)).
\]

To construct explicit generators for the \(\KR\)\nb-theory of
spheres, we shall need an elementary auxiliary result used in the
proof in~\cite{Roe:Paschke_real}.  Let~\(\mathcal{H}_A\) be the
standard graded ``real'' Hilbert \(A\)\nb-module
\(\ell^2(\N,\C) \otimes A\); the ``real'' involution is the tensor
product of complex conjugation on \(\ell^2(\N,\C)\) and the given
``real'' involution on~\(A\), and the \(\Z/2\)\nb-grading is the
tensor products of the \(\Z/2\)\nb-grading on \(\ell^2(\N,\C)\)
induced by the parity operator and the given one on~\(A\).  A cycle
for \(\KKR_0(\C, A \hot \Cliff_{r,s})\) is homotopic to one of the
form \((\mathcal{H}_A \hot \Cliff_{r,s}, 1, F)\), where~\(1\) denotes
the left action of~\(\C\) by scalar multiplication, and
\(F\in\Bound(\mathcal{H}_A)\hot \Cliff_{r,s}\) is real, odd and
selfadjoint and satisfies
\(F^2 -1 \in \Comp(\mathcal{H}_A) \hot \Cliff_{r,s}\).  Let
\(\Comp= \Comp(\ell^2(\N,\C))\), identify \(\Comp(\mathcal{H}_A)\)
with \(\Comp \otimes A \cong \Comp \hot A\), and let
\(\Mult(\Comp \otimes A)\) denote the stable multiplier algebra
of~\(A\).  Then \(\Bound(\mathcal{H}_A \hot \Cliff_{r,s})\) is
isomorphic to \(\Mult(\Comp \otimes A) \hot \Cliff_{r,s}\).  Let
\(\mathcal{Q}(\Comp \otimes A) \defeq \Mult(\Comp \otimes A)/\Comp
\otimes A\).  Then~\(F\) is mapped to a ``real'', odd, selfadjoint
unitary in \(\mathcal{Q}(\Comp \otimes A)\hot \Cliff_{r,s}\).
Conversely, any ``real'', odd, selfadjoint unitary in
\(\mathcal{Q}(\Comp \otimes A)\hot \Cliff_{r,s}\)
lifts to an operator
\(F \in \Bound(\mathcal{H}_A \hot \Cliff_{r,s})\) such that
\((\mathcal{H}_A \hot \Cliff_{r,s}, 1, F)\) is a cycle for
\(\KKR_0(\C,A \hot \Cliff_{r,s})\).  Thus we get an obvious
surjective map from
\(\FU(\mathcal{Q}(\Comp \otimes A)\hot \Cliff_{r,s})\) onto
\(\KKR_0(\C,A \hot \Cliff_{r,s})\).  Since~\(\mathcal{H}_A\)
has infinite multiplicity,
\(\Bound(\mathcal{H}_A \hot \Cliff_{r,s}) \cong
\Mat_n\bigl(\Bound(\mathcal{H}_A \hot \Cliff_{r,s})\bigr)\), and the
same holds for the quotient by the compact operators.  Thus the
construction above also defines a map
\[
  \bigsqcup_{n=1}^\infty
  \FU_n(\mathcal{Q}(\Comp \otimes A)\hot \Cliff_{r,s})
  \to \KKR_0(\C,A \hot \Cliff_{r,s}).
\]

\begin{lemma}
  \label{lem:Daele_vs_Kasparov}
  The map above induces a group isomorphism
  \(\DK(\mathcal{Q}(\Comp \otimes A)\hot \Cliff_{r,s}) \congto
  \KKR_0(\C,A \hot \Cliff_{r,s})\).
\end{lemma}

\begin{proof}
  Our map is a semigroup homomorphism because the sum in \(\DK\) and
  \(\KKR\) is defined as the direct sum.  As such, it induces a
  group homomorphism on the Grothendieck group, which we may
  restrict to the subgroup \(\DK(A)\).  There is a real, odd,
  selfadjoint unitary operator
  \(F_0 \in \Bound(\mathcal{H}_A \hot \Cliff_{r,s})\), and any such
  operator defines a degenerate Kasparov cycle.  Such degenerate
  cycles represent zero in KK-theory.  If
  \(F\in \mathcal{FU}(\mathcal{Q}(\Comp \otimes A)\hot
  \Cliff_{r,s})\), then \([F] - [F_0] \in \DK(A)\) is mapped to the
  same \(\KK\)-class as~\(F\).  Thus the induced map on \(\DK(A)\)
  remains surjective.  Roe shows in~\cite{Roe:Paschke_real} that it
  is also injective; the main reason for this is that ``homotopic''
  KK-cycles become ``operator homotopic'' after adding degenerate
  cycles.
\end{proof}

Roe further combines the isomorphism in
Lemma~\ref{lem:Daele_vs_Kasparov} with the boundary map in the long
exact sequence for~\(\DK\) for the extension of \(\Cst\)\nb-algebras
\[
  \Comp \otimes A \hot \Cliff_{r,s}
  \into \Mult(\Comp \otimes A) \hot \Cliff_{r,s}
  \prto \mathcal{Q}(\Comp \otimes A) \hot \Cliff_{r,s}
\]
to prove Proposition~\ref{pro:Daele_vs_KK}.  The latter boundary map
is an isomorphism because van Daele's \(\K\)\nb-theory, like ordinary
\(\K\)\nb-theory, vanishes for stable multiplier algebras.

Now we combine Lemma~\ref{lem:Daele_vs_Kasparov} with Kasparov's
proof of Bott periodicity (see \cite{Kasparov:Operator_K}*{p.~545f
  and Theorem~7}) in bivariant KK-theory.  Fix \(a,b\in\N\); at some
point, we will assume \(a\ge1\), but for now we do not need this.
Let
\[
  D\defeq \Cont_0(\R^{a,b}) \otimes \Cliff_{a,b}
  = \Cont_0(\R^{a,b}) \hot \Cliff_{a,b}.
\]
Kasparov constructs an element in \(\KKR_0(\C,D)\) and proves that
it is invertible by defining an inverse and computing the Kasparov
products.  The definition of Kasparov's KK-class is based on the
unbounded continuous function
\begin{equation}
  \label{eq:tildebeta}
  \tilde{\beta}_{a,b} \colon  \R^{a,b} \to \Cliff_{a,b},\qquad
  (x_1,\ldots,x_{a+b}) \mapsto \sum_{i=1}^{a+b}x_i\gamma_i.
\end{equation}

\begin{lemma}
  \(\tilde{\beta}_{a,b}\) is selfadjoint, odd and real, and
  \(\tilde{\beta}_{a,b}(x)^2 = \norm{x}^2\).
\end{lemma}

\begin{proof}
  First, \(\tilde{\beta}_{a,b}\) is selfadjoint because
  \(x_i^* = x_i\) and \(\gamma_i^* = \gamma_i\) for all~\(i\).  It
  is odd because all~\(\gamma_i\) are odd and the grading on
  \(\Cont_0(\R^{a,b})\) is trivial.  Finally,
  \(\tilde{\beta}_{a,b}\) is real because~\(\gamma_i\) is real
  whenever~\(x_i\) is real and~\(\gamma_i\) is imaginary
  whenever~\(x_i\) is imaginary.  The formula
  \(\tilde{\beta}_{a,b}(x)^2 = \norm{x}^2\) holds because
  \(\gamma_i^2=1 \) for all \(1 \leq i \leq a+b\) and
  the~\(\gamma_i\) anticommute.
\end{proof}  

The bounded transform of~\(\tilde{\beta}_{a,b}\) is
\[
  \tilde\beta'_{a,b}\colon  \R^{a,b} \to \Cliff_{a,b}, \qquad x\mapsto
  \tilde{\beta}_{a,b}(x) / (1 + \tilde{\beta}_{a,b}(x)^2)^{1/2}.
\]
It satisfies
\(1-(\tilde\beta'_{a,b})^2= (1+\norm{x}^2)^{-1/2} \in \Comp(D)\) and
is odd, selfadjoint and real because~\(\tilde{\beta}_{a,b}\) is.
Hence \([\tilde\beta'_{a,b}] \in \KKR_0(\C, D)\).  Kasparov proves
that \([\tilde\beta'_{a,b}]\) generates \(\KKR_0(\C, D) \cong \Z\).
  
To plug \([\tilde\beta'_{a,b}]\) into the isomorphism in
Lemma~\ref{lem:Daele_vs_Kasparov}, we choose a degenerate
cycle~\(F_0\).  Adding~\(F_0\) replaces the underlying Hilbert
module~\(D\) of~\([\tilde\beta'_{a,b}]\) by~\(\mathcal{H}_D\).  Thus
\(\tilde\beta'_{a,b}\oplus F_0\) and~\(F_0\) are elements of
\(\mathcal{FU}\bigl(\mathcal{Q}(\Cont_0(\R^{a,b}, \Comp \otimes
\Cliff_{a,b}))\bigr)\) and \([\tilde\beta'_{a,b} \oplus F_0] - [F_0]\)
represents a class in
\(\DK\bigl(\mathcal{Q}(\Cont_0(\R^{a,b}, \Comp \otimes
\Cliff_{a,b}))\bigr)\).  Since degenerate KK-cycles represent zero,
\([\tilde\beta'_{a,b} \oplus F_0] - [F_0]\) is a generator for
\(\DK(\mathcal{Q}(\Cont_0(\R^{a,b}) \otimes \Comp \otimes
\Cliff_{a,b})) \cong \Z\).  Here the choice of~\(F_0\) does not matter.

We may also interpret the passage from~\(\tilde\beta^{a,b}\)
to~\(\tilde\beta'^{a,b}\)
as replacing~\(\R^{a,b}\) by
\[
  B^{a,b} \defeq \setgiven{x\in\R^{a,b}}{\norm{x}<1},
\]
which is homeomorphic to~\(\R^{a,b}\) through the map
\(x\mapsto x/(1+\norm{x}^2)^{1/2}\).  Let~\(\bar{B}^{a,b}\) be the
closure of~\(B^{a,b}\) in~\(\R^{a,b}\).  The boundary
\(\bar{B}^{a,b} \setminus B^{a,b}\) is the \(a+b-1\)-dimensional
unit sphere \(\Sphere^{a,b} \subseteq \R^{a,b}\) with the ``real''
involution of~\(\R^{a,b}\).  There is a canonical embedding
\(\Cont(\bar{B}^{a,b}) \otimes \Cliff_{a,b} \hookrightarrow
\Contb(B^{a,b}) \otimes \Cliff_{a,b} = \Mult(\Cont_0(B^{a,b})
\otimes \Cliff_{a,b})\).  It induces a canonical embedding
\(\Cont(\Sphere^{a,b}) \otimes \Cliff_{a,b} \hookrightarrow
\Mult(\Cont_0(\R^{a,b}) \otimes \Cliff_{a,b})/ (\Cont_0(\R^{a,b})
\otimes \Cliff_{a,b})\).  Its image contains the image
of~\(\tilde\beta'_{a,b}\) in this quotient.  Namely,
\([\tilde\beta'_{a,b}]\) comes from
\(\beta_{a,b} \defeq \tilde\beta_{a,b}|_{\Sphere^{a,b}} \in
\mathcal{FU}(\Cont(\Sphere^{a,b}) \otimes \Cliff_{a,b})\); this is
the same~\(\beta_{a,b}\) as in~\eqref{eq:beta}.

We assume from now on that \(a\ge 1\).  Then the first Clifford
generator \(\gamma_1\in \Cliff_{a,b}\) is an odd selfadjoint real
unitary.  The constant function with value~\(\gamma_1\) belongs to
\(\mathcal{FU}(\Cont(\bar{B}^{a,b}) \otimes \Cliff_{a,b})\) and is
mapped to a degenerate \(\KK\)-cycle for \(\KKR_0(\C,D)\).  So we
map pick~\(F_0\) to be a direct sum of countably many copies
of~\(\gamma_1\) above.

\begin{lemma}
  \label{lem:lift_beta_to_sphere}
  Let \(a\ge1\).  The class of
  \([\beta_{a,b}] - [\gamma_1] \in \DK(\Cont(\Sphere^{a,b}) \otimes
  \Cliff_{a,b})\) is mapped to a generator of \(\KKR_0(\C,D)\).
\end{lemma}

\begin{proof}
  The inclusion of \(\Cont(\Sphere^{a,b}) \otimes \Cliff_{a,b}\)
  into
  \(\mathcal{Q}(\Cont_0(\R^{a,b}) \otimes \Comp \otimes
  \Cliff_{a,b})\) is not unital.  The map on van Daele
  \(\K\)\nb-theory that it induces sends \([e] - [f]\) for
  \(e,f\in \FU(\Cont(\Sphere^{a,b}) \otimes \Cliff_{a,b})\) to
  \([e \oplus F_0] - [f \oplus F_0]\).  For our choice of~\(F_0\)
  above, it therefore sends \([\beta_{a,b}] - [\gamma_1]\) to
  \([\beta_{a,b} \oplus F_0] - [F_0] \in
  \DK\bigl(\mathcal{Q}(\Cont_0(\R^{a,b}) \otimes \Comp \otimes
  \Cliff_{a,b})\bigr)\).  We already know that Roe's isomorphism
  maps the latter to a generator of \(\KKR_0(\C,D)\).
\end{proof}

Since \(a\ge1\), the two poles \(S\) and~\(N\) in~\ref{eq:poles} are
fixed by the ``real'' involution on~\(\Sphere^{a,b}\).  We choose
\(N\in\Sphere^{a,b}\) as a point at infinity and identify
\(\Sphere^{a,b}\setminus\{N\} \cong \R^{a-1,b}\) by stereographic
projection.  The \(\Cst\)\nb-algebra extension
\[
  \begin{tikzcd}
    \Cont_0(\R^{a-1,b}) \ar[r,>->]&
    \Cont(\Sphere^{a,b}) \ar[r, "\ev_N", ->>] &
    \R
  \end{tikzcd}
\]
splits by taking constant functions.  Since van Daele's \(\K\)\nb-theory is
split exact
\[
  \DK(\Cont(\Sphere^{a,b}) \otimes \Cliff_{r,s})
  \cong \DK(\Cont_0(\R^{a-1,b}) \otimes \Cliff_{r,s})
  \oplus \DK(\Cliff_{r,s}).
\]

\begin{proposition}
  There is a Bott periodicity isomorphism
  \[
    \DK(\Cont_0(\R^{a-1,b}) \otimes \Cliff_{a,b})
    \cong \KKR_0(\C,\Cont_0(\R^{a-1,b}) \otimes \Cliff_{a-1,b})
    \cong \Z.
  \]
  Under the two isomorphisms above,
  \([\beta_{a,b}] - [\gamma_1] \in \DK(\Cont(\Sphere^{a,b}) \otimes
  \Cliff_{a,b})\) is mapped to
  \((1,0) \in \Z \oplus \DK(\Cliff_{a,b})\).
\end{proposition}

\begin{proof}
  The value of~\(\beta_{a,b}\) at~\(N\) is~\(\gamma_1\).  So
  evaluation at~\(N\) maps~\([\beta_{a,b}] -[\gamma_1]\) to zero.
  We compute \(\DK(\Cont_0(\R^{a-1,b}) \otimes \Cliff_{a,b})\)
  through the long exact sequences for~\(\DK\) applied to the
  morphism of \(\Cst\)\nb-algebra extensions
  \[
    \begin{tikzcd}
      \Cont_0(\R^{a,b}) \arrow[r, >->] \ar[d,hookrightarrow] &
      \Cont(\bar{B}^{a,b}\backslash{\{N\}})
      \arrow[r, ->>] \ar[d, hook] &
      \Cont_0(\Sphere^{a,b}\backslash\{N\})
      \ar[d, hook]\\
      \Cont_0(\R^{a,b}) \otimes \Comp \arrow[r, >->] &
      \Mult(\Cont_0(\R^{a,b}) \otimes \Comp)
      \arrow[r, ->>] &
      \mathcal{Q}(\Cont_0(\R^{a,b}) \otimes \Comp).
    \end{tikzcd}
  \]
  The vertical arrow
  \(\Cont(\bar{B}^{a,b}\backslash{\{N\}}) \hookrightarrow
  \Mult(\Cont_0(\R^{a,b}) \otimes \Comp)\) combines the map from
  bounded continuous functions on~\(\R^{a,b}\) to multipliers of
  \(\Cont_0(\R^{a,b})\) and the corner embedding into~\(\Comp\).  It
  induces an isomorphism on~\(\DK\) because the latter is homotopy
  invariant, stable under tensoring with~\(\Comp\) and vanishes on
  stable multiplier algebras.  The left vertical map induces an
  isomorphism on~\(\DK\) as well.  Therefore, the vertical map on
  the quotients also induces an isomorphism
  \[
    \DK(\Cont_0(\Sphere^{a,b}\backslash\{N\}))
    \cong \DK(\mathcal{Q}(\Cont_0(\R^{a,b}) \otimes \Comp))
    \cong \KKR_0(\C,D) \cong \Z.
  \]
  Lemma~\ref{lem:lift_beta_to_sphere} shows that the image of
  \([\beta_{a,b}] -[\gamma_1]\) is mapped to a generator of~\(\Z\)
  under these isomorphisms.
\end{proof}

How about generators for \(\DK(\Cont_0(\R^{a-1,b}) \otimes \Cliff_{r,s})\)
for \((r,s) \neq (a,b)\)?  By Bott periodicity, it is no loss of
generality to assume \(r= a+c\) and \(s= b+d\) with \(c,d\ge0\), and
this group is isomorphic to \(\KKR_0(\C, \Cliff_{c,d})\).  The
isomorphism in Kasparov theory is given simply by exterior product.
In order to compute these exterior products explicitly, we simplify
the cycles for \(\KKR_0(\C, \Cliff_{c,d})\).  The following results
may well be known, but the authors are not aware of a reference
where generators for the \(\KO\)\nb-theory of a point are worked out
as cycles for \(\KKR_0(\C, \Cliff_{c,d})\).  Of course, our results
will recover the description of \(\KO^*(\pt)\) through Clifford
modules by Atiyah--Bott--Shapiro~\cite{Atiyah:K_Reality}.

\begin{lemma}
  \label{lem:KK_to_Cliff}
  Any cycle for \(\KKR_0(\C, \Cliff_{c,d})\) is homotopic to one
  with~\(\C\) acting by scalar multiplication and Fredholm operator
  equal to~\(0\).  Isomorphism classes of such cycles are in
  bijection with finitely generated modules over the real subalgebra
  \((\Cliff_{c+1,d})_\R\) in~\(\Cliff_{c+1,d}\).
\end{lemma}

\begin{proof}
  Let \((\mathcal{H},\varphi,F)\) be a cycle for
  \(\KKR_0(\C,\Cliff_{c,d})\).  First, we may replace it by a
  homotopic one where~\(\C\) acts just by scalar multiplication.
  Then we may use functional calculus for~\(F\) to arrange that the
  spectrum of~\(F\) consists only of \(\{0,1,-1\}\).  The direct
  summand where~\(F\) has spectrum \(\pm1\) is unitary and thus
  gives a degenerate cycle.  Removing that piece gives a cycle with
  \(F=0\), still with~\(\C\) acting by scalar multiplication.  Thus
  the only remaining data is the underlying \(\Z/2\)-graded ``real''
  Hilbert \(\Cliff_{c,d}\)\nb-module, which we still denote
  by~\(\mathcal{H}\).  To give a \(\KK\)-cycle with \(F=0\), the
  identity operator on~\(\mathcal{H}\) must be compact.

  Next we claim that \(\mathcal{H} \cong p\cdot \Cliff_{c,d}^{2n}\)
  for some real, even projection
  \(p\in\hat\Mat_{2n}(\Cliff_{c,d})\); here half of the summands in
  \(\Cliff_{c,d}^{2n}\) have the flipped grading.  First, the
  Kasparov stabilisation theorem implies that there is a real, even
  unitary
  \(\mathcal{H} \oplus (\Cliff_{c,d})^\infty \cong
  (\Cliff_{c,d})^\infty\), where \((\Cliff_{c,d})^\infty\) denotes
  the standard ``real'' graded Hilbert \(\Cliff_{c,d}\)-module, with
  the \(\Z/2\)-grading where half of the summands carry the flipped
  grading.  This gives a real, even projection
  \(p_0 \in \Bound((\Cliff_{c,d})^\infty)\) with
  \(\mathcal{H} \cong p_0 (\Cliff_{c,d})^\infty\).  Since the
  identity on~\(\mathcal{H}\) is compact,
  \(p_0 \in \Comp((\Cliff_{c,d})^\infty)\).  Then~\(p_0\) is
  Murray--von Neumann equivalent --~with even real partial
  isometries~-- to a nearby projection
  \(p\in \hat\Mat_{2n}(\Cliff_{c,d})\).  This implies an isomorphism
  \(\mathcal{H} \cong p\cdot \Cliff_{c,d}^{2n}\) of ``real''
  \(\Z/2\)-graded Hilbert modules.

  It is well known that any idempotent in a \(\Cst\)\nb-algebra is
  Murray--von Neumann equivalent to a projection.  The proof is
  explicit, using the functional calculus.  Therefore, if the
  idempotent we start with is even and real, then so are the
  equivalent projection and the Murray--von Neumann equivalence
  between the two.  Therefore, the isomorphism class
  of~\(\mathcal{H}\) as a ``real'' graded Hilbert module is still
  captured by the Murray--von Neumann equivalence class of~\(p\) as
  a real, even idempotent element.  Let~\(S\) be the ring of even,
  real elements in \(\hat\Mat_2(\Cliff_{c,d})\).  What we end up
  with is that the possible isomorphism classes of~\(\mathcal{H}\)
  are in bijection with Murray--von Neumann equivalence classes of
  idempotents in matrix algebras over~\(S\).  These are, in turn, in
  bijection to isomorphism classes of finitely generated projective
  modules over~\(S\).  The subalgebra of even elements of
  \(\hat\Mat_2(\Cliff_{c,d})\) is isomorphic to the crossed product
  \(\Cliff_{c,d}\rtimes \Z/2\) --~this is an easy special case of
  the Green--Julg Theorem (see~\cite{Julg:K_equivariante}),
  identifying a crossed product for an
  action of a compact group~\(G\) on a \(\Cst\)\nb-algebra~\(A\)
  with the fixed-point algebra of the diagonal \(G\)\nb-action on
  \(A\otimes \Comp(L^2 G)\).  The nontrivial element of~\(\Z/2\)
  gives an extra Clifford generator that commutes with even and
  anticommutes with odd elements of~\(\Cliff_{c,d}\).  Thus the even
  part of \(\hat\Mat_2(\Cliff_{c,d})\) is isomorphic to
  \(\Cliff_{c+1,d}\).  This implies \(S\cong (\Cliff_{c+1,d})_\R\).
  This \(\R\)\nb-algebra is semisimple because its complexification
  is a sum of matrix algebras.  So all finitely generated modules
  over it are projective.
\end{proof}

\begin{remark}
  The lemma only describes isomorphism classes of certain special
  cycles for \(\KKR_0(\C,\Cliff_{c,d})\), it does not compute this
  group.  We do not need this computation, but sketch it anyway.
  Consider a degenerate cycle for \(\KKR_0(\C,\Cliff_{c,d})\) that
  is also finitely generated.  The operator~\(F\) on it is real and
  odd with \(F^2=1\) and commutes with~\(\Cliff_{c,d}\).
  Multiplying~\(F\) with the grading gives a real, odd operator that
  anticommutes with the generators of~\(\Cliff_{c,d}\) and has
  square~\(-1\).  This shows that a \((\Cliff_{c+1,d})_\R\)-module
  admits an operator~\(F\) that makes it a degenerate cycle for
  \(\KKR_0(\C,\Cliff_{c,d})\) if and only if the module structure
  extends to \((\Cliff_{c+1,d+1})_\R\).  Call two
  \((\Cliff_{c+1,d})_\R\)-modules \emph{stably isomorphic} if they
  become isomorphic after adding restrictions of
  \((\Cliff_{c+1,d+1})_\R\)-modules to them.  Since the direct sum
  of \((\Cliff_{c+1,d})_\R\)-modules corresponds to the direct sum
  of Kasparov cycles, stably isomorphic
  \((\Cliff_{c+1,d})_\R\)-modules give the same class in
  \(\KKR_0(\C,\Cliff_{c,d})\).  Now
  Atiyah--Bott--Shapiro~\cite{Atiyah-Bott-Shapiro:Clifford} have
  computed the stable isomorphism classes of
  \((\Cliff_{c+1,d})_\R\)-modules and found that they give the
  \(\KO\)\nb-theory of the point.  As a result, two
  \((\Cliff_{c+1,d})_\R\)-modules give the same class in
  \(\KKR_0(\C,\Cliff_{c,d})\) if \emph{and only if} they are stably
  isomorphic.
\end{remark}

Let us look at the generators of
\(\KKR_0(\C,\Cliff_{c,d}) \cong \Z/2\) for
\(d-c\equiv 1,2 \bmod 8\).  The simplest choice here is to take
\(c=0\) and \(d=1,2\).  It follows from Lemma~\ref{lem:KK_to_Cliff}
that the generator of \(\KKR_0(\C,\Cliff_{c,d}) \cong \Z/2\)
corresponds to a \((\Cliff_{c+1,d})_\R\)-module.  Since the direct
sum of modules becomes the sum in \(\KKR\)-theory, we may pick a
simple module.  For \(c=0\) and \(d=1\), we get
\((\Cliff_{1,1})_\R = \Mat_2(\R)\).  Up to isomorphism, there is a
unique two-dimensional simple module.  Turn \(\Cliff_{0,1}\) into a
\(\Z/2\)-graded Hilbert module over itself and give it the operator
\(F=0\).  The constructions in the proof of
Lemma~\ref{lem:KK_to_Cliff} turn this into a two-dimensional
\((\Cliff_{1,1})_\R\)-module.  So~\(\Cliff_{0,1}\) with \(F=0\)
represents the generator of \(\KKR_0(\C,\Cliff_{0,1}) \cong \Z/2\).

Next, let \(c=0\) and \(d=2\).  Then
\((\Cliff_{1,2})_\R = \Mat_2(\C)\).  It has~\(\C^2\) as its unique
simple module.  So any module of dimension~\(4\) over~\(\R\) is
simple.  Since \(\Cliff_{0,2}\) as a graded Hilbert module over
itself also yields a \(4\)\nb-dimensional module over
\((\Cliff_{1,2})_\R\), the latter is simple.
Thus~\(\Cliff_{0,2}\) with \(F=0\) represents the generator of
\(\KKR_0(\C,\Cliff_{0,2}) \cong \Z/2\).

Having described the generators of \(\KKR_0(\C,\Cliff_{0,d})\) for
\(d=1,2\) in this simple form, computing the exterior product
becomes a trivial matter:

\begin{proposition}
  \label{pro:other_sphere_generators}
  Let \(d\in\{1,2\}\).  Then the isomorphism
  \begin{align*}
    \DK(\Cont(\Sphere^{a,b}) \otimes \Cliff_{a,b+d})
    &\cong \DK(\Cont_0(\R^{a-1,b}) \otimes \Cliff_{a,b+d})
      \oplus \DK(\Cliff_{a,b+d})
    \\&\cong \Z/2 \oplus \DK(\Cliff_{a,b+d})
  \end{align*}
  maps \([\beta_{a,b}] - [\gamma_1]\) to~\((1,0)\).
\end{proposition}
  
\begin{proof}
  First we recall a general formula for exterior products in a
  simple special case.  Let \((\mathcal{E},\varphi,F)\) be a cycle
  for \(\KKR_0(A,B)\) for some ``real'' \(\Cst\)\nb-algebras
  \(A,B\).  Let \((\mathcal{H},1,0)\) be a cycle for
  \(\KKR_0(\C,\Cliff_{c,d})\).  Then their exterior product is
  represented by the triple
  \((\mathcal{E} \hot \mathcal{H}, \varphi \hot 1, F\hot 1)\); this
  works because~\(\mathcal{H}\) has Fredholm operator~\(0\).  We now
  apply this to Kasparov's Bott periodicity generator
  \([\tilde\beta'_{a,b}]\).  Then we get the \(\Z/2\)-graded Hilbert
  \(\Cont_0(\R^{a,b}) \otimes \Cliff_{a+c,b+d}\)-module
  \(\Cont_0(\R^{a,b}) \otimes \Cliff_{a,b} \hot \mathcal{H}\)
  with~\(\C\) acting by scalar multiplication and with the Fredholm
  operator \(\tilde\beta'_{a,b} \hot 1\).  The same arguments as
  above show that the isomorphism to van Daele's \(\K\)\nb-theory
  transfers this Kasparov cycle to
  \([\beta_{a,b}] - [\gamma_1]\) in
  \(\DK\bigl(\Cont(\Sphere^{a,b}) \otimes \Cliff_{a,b} \hot
  \Comp(\mathcal{H})\bigr)\).  The latter is \(\Z/2\)-equivariantly
  Morita equivalent to
  \(\DK(\Cont(\Sphere^{a,b}) \otimes\Cliff_{a+c,b+d})\) because
  \(\Comp(\mathcal{H})\) is \(\Z/2\)-equivariantly Morita equivalent
  to~\(\Cliff_{c,d}\).  The isomorphism in van Daele's
  \(\K\)\nb-theory induced by this Morita equivalence maps
  \([\beta_{a,b}] - [\gamma_1] \in \DK\bigl(\Cont(\Sphere^{a,b})
  \otimes \Cliff_{a,b} \hot \Comp(\mathcal{H})\bigr)\) to the
  element in \(\DK(\Cont(\Sphere^{a,b}) \otimes \Cliff_{a+c,b+d})\)
  that comes from \((\mathcal{H},1,0)\in\KKR_0(\C,\Cliff_{c,d})\).

  Now we specialise to the case where \(\mathcal{H} = \Cliff_{0,d}\)
  with the obvious structure of \(\Z/2\)-graded ``real'' Hilbert
  \(\Cliff_{0,d}\)-module.  Then
  \(\Comp(\mathcal{H}) \cong \Cliff_{0,d}\) canonically, and there
  is no need to invoke the Morita invariance of van Daele's
  \(\K\)\nb-theory.  The result is just to view \(\beta_{a,b}\) as
  taking values in \(\Cliff_{a,b+d} \supseteq \Cliff_{a,b}\).
\end{proof}

We do not describe the generator for
\(\KKR_0(\C,\Cliff_{0,4}) \cong \Z\).  Since
\(\mathcal{H}=\Cliff_{0,4}\) itself represents~\(0\) in \(\KKR\), a
nontrivial Morita equivalence is needed here.

\section{Some canonical maps in \texorpdfstring{$\KR$\nb-}{KR-}theory}
\label{sec:KR-maps}

We have described explicit elements in the van Daele \(\K\)\nb-theory of the
spheres~\(\Sphere^{a,b}\) for \(a,b\in\N\).  Topological phases are,
however, described by the \(\K\)\nb-theory of the ``real'' torus~\(\T^d\),
where~\(d\) is the dimension of the physical system.
One way to transfer \(\K\)\nb-theory classes from one space to another is
the pullback functoriality.  In Section~\ref{sec:Hamiltonians}, we
are going to use this to pull our \(\K\)\nb-theory generators
on~\(\Sphere^{1,d}\) back to \(\K\)\nb-theory classes on~\(\T^d\) along
certain maps \(\T^d \to \Sphere^{1,d}\).

We now describe this pull back functoriality in detail.  Let
\((X,\rinv_X)\) and \((Y,\rinv_Y)\) be ``real'' locally compact
spaces and let \(b\colon X\to Y\) be a proper continuous map that is
``real'' in the sense that \(\rinv_Y \circ b = b\circ \rinv_X\).
Such a map induces a ``real'' \Star{}homomorphism
\begin{equation}
  \label{eq:def_bstar}
  b^*\colon \Cont_0(Y) \to \Cont_0(X),\qquad
  \psi\mapsto \psi\circ b,
\end{equation}
which for \(c,d\in\N\) induces maps in van Daele's \(\K\)\nb-theory
\[
  b^*\colon \DK(\Cont_0(Y) \otimes \Cliff_{c,d})
  \to \DK(\Cont_0(X) \otimes \Cliff_{c,d}).
\]
We denote it by~\(b^*\) as well because no confusion should be
possible.

There are other ways to transfer \(\K\)\nb-theory classes between spaces.  A
famous example is the Atiyah--Singer index map for a family of
elliptic differential operators.  There are two ways to compute this
index map, one analytic and one topological.  We are going to use
the topological approach.  One of its ingredients is functoriality
for open inclusions: let \(U\subseteq X\) be an open subset with
\(\rinv_X(U) = U\), so that the ``real'' structure~\(\rinv_X\)
restricts to one on~\(U\).  Then extension by zero defines a ``real''
\Star{}homomorphism
\[
  \iota_U\colon \Cont_0(U) \to \Cont_0(X),
\]
which for \(c,d\in\N\) induces maps in van Daele's \(\K\)\nb-theory
\[
  \iota_{U,*}\colon \DK(\Cont_0(U) \otimes \Cliff_{c,d})
  \to \DK(\Cont_0(X) \otimes \Cliff_{c,d}).
\]

The other ingredient of the topological index map is the Thom
isomorphism.  Let \(E \prto X\) be a \(\Z/2\)-equivariant vector
bundle.  Let \(\Cliff_E\) be the Clifford algebra bundle of~\(E\);
this is a locally trivial bundle of finite-dimensional
\(\Z/2\)-graded, real \(\Cst\)\nb-algebras over~\(X\).  Each fibre
\(\Cont_0(E_x) \otimes \Cliff(E_x)\) carries a canonical Kasparov
generator~\(\tilde\beta'_{E_x}\).  Letting~\(x\) vary, these combine
to a class
\begin{equation}
  \label{eq:Thom_with_bundle}
  [\tilde\beta'_E] \in
  \KKR_0(\Cont_0(X),\Cont_0(\abs{E}) \otimes \Cliff_E).  
\end{equation}
This class is also invertible and produces a raw form of the Thom
isomorphism.

\begin{definition}
  \label{def:KR-orientation}
  A \emph{\(\KR\)\nb-orientation} on a \(\Z/2\)-equivariant vector
  bundle \(E\prto X\) is a \(\Z/2\)-graded, ``real'',
  \(\Cont_0(X)\)-linear Morita equivalence between
  \(\Cont_0(X,\Cliff_E)\) and \(\Cont_0(X) \otimes \Cliff_{a,b}\)
  for some \(a,b\in\N\).  That is, it is a full Hilbert
  bimodule~\(M\) over \(\Cont_0(X,\Cliff_E)\) and
  \(\Cont_0(X) \otimes \Cliff_{a,b}\) with a compatible ``real''
  structure and \(\Z/2\)-action and with the extra property that the
  two actions of \(\Cont_0(X)\) by multiplication on the left and
  right are the same.  A vector bundle is called
  \emph{\(\KR\)\nb-orientable} if it has such a
  \(\KR\)\nb-orientation.  We call \(a-b \bmod 8 \in \Z/8\) the
  \emph{\(\KR\)\nb-dimension} \(\dim_\KR E\) of~\(E\).
\end{definition}

Morita equivalence for ``real'' \(\Z/2\)-graded \(\Cst\)\nb-algebras
with an action of a locally compact groupoid is explored by
Moutuou~\cite{Moutuou:Graded_Brauer}.  Our definition is the special
case where the groupoid is the space~\(X\) with only identity
arrows.  The idea to describe \(\K\)\nb-theory orientations through
Morita equivalence goes back to Plymen~\cite{Plymen:Morita_Spinors}.

Two Clifford algebras \(\Cliff_{a,b}\) and~\(\Cliff_{a',b'}\) are
Morita equivalent as ``real'' \(\Z/2\)-graded \(\Cst\)\nb-algebras
if and only if \(a-b \equiv a'-b' \bmod 8\); this is part of the
computation of the Brauer group of the point in
\cite{Moutuou:Graded_Brauer}*{Appendix~A}.  This shows that the
\(\KR\)\nb-dimension is well defined and that it suffices to
consider, say, the cases \(b=0\), \(a\in \{0,1,\dotsc,7\}\) in
Definition~\ref{def:KR-orientation}.  A trivial vector bundle
\(X\times \R^{a,b} \prto X\) is, of course, \(\KR\)\nb-oriented
because \(\Cliff_{X\times\R^{a,b}} = X\times \Cliff_{a,b}\).  Its
\(\KR\)\nb-dimension is \(a-b \bmod 8\).

If \(a=b=0\), then a \(\KR\)\nb-orientation is a \(\Z/2\)-graded, ``real''
Morita equivalence between \(\Cont_0(X,\Cliff_E)\) and
\(\Cont_0(X)\).  This is equivalent to a \(\Z/2\)-graded ``real''
Hilbert \(\Cont_0(X)\)-module~\(S\) with a \(\Cont_0(X)\)-linear
isomorphism \(\Comp(S) \cong \Cont_0(X,\Cliff_E)\).  This
forces~\(S\) to be the space of \(\Cont_0\)-sections of a complex
vector bundle over~\(X\), equipped with a \(\Z/2\)-grading and a
``real'' involution.  Then the isomorphism
\(\Comp(S) \cong \Cont_0(X,\Cliff_E)\) means that the fibre~\(S_x\)
of this vector bundle carries an irreducible representation of
\(\Cliff_{E_x}\).  Thus~\(S\) is a ``real'' spinor bundle for~\(E\).
Allowing \(a,b\neq0\) gives appropriate analogues of the spinor
bundle for ``real'' vector bundles of all dimensions.

The \(\KR\)\nb-orientation induces a \(\KKR\)-equivalence because
\(\KKR\) is Morita invariant.  Together with the
\(\KKR\)-equivalence in~\eqref{eq:Thom_with_bundle}, this gives a
\(\KKR\)-equivalence
\[
  \tau_E \in
  \KKR_0\bigl(\Cont_0(X),\Cont_0(\abs{E}) \otimes \Cliff_{a,b}\bigr)
  \cong \KKR_{-\dim_\KR E}\bigl(\Cont_0(X),\Cont_0(\abs{E})\bigr),
\]
called the \emph{Thom isomorphism class}.  It induces Thom
isomorphisms in \(\K\)\nb-theory
\[
  \KR^n(X) \cong \KR^{n+\dim_\KR E}(E) \qquad
  \text{for }n\in\Z/8.
\]

\begin{example}
  \label{exa:complex_oriented}
  Let \(E\prto X\) be a complex vector bundles with a ``real''
  structure.  A Thom isomorphism for~\(E\) is defined in
  \cite{Atiyah:K_Reality}*{Theorem~2.4}.  It is analogous to the
  Thom isomorphism for complex vector bundles in ordinary complex
  \(\K\)\nb-theory.  Atiyah's Thom isomorphism may also be defined
  using a \(\KR\)-orientation of dimension~\(0\) as defined above.
  Namely, the sum of the complex exterior powers of~\(E\) provides a
  spinor bundle for~\(E\), which also carries a canonical ``real''
  structure to become a \(\KR\)\nb-orientation of
  \(\KR\)\nb-dimension~\(0\) as in
  Definition~\ref{def:KR-orientation}.
\end{example}

\begin{lemma}
  \label{lem:add_KR-orientations}
  Let \(V_1,V_2\) be two \(\Z/2\)-equivariant vector bundles
  over~\(X\).  If two of the vector bundles
  \(V_1,V_2,V_1\oplus V_2\) are \(\KR\)\nb-oriented, then this induces a
  canonical \(\KR\)\nb-orientation on the third one, such that
  \[
    \dim_\KR V_1 + \dim_\KR V_2 = \dim_\KR (V_1 \oplus V_2).
  \]
\end{lemma}

\begin{proof}
  A \(\Cont_0(X)\)-linear Morita equivalence is the same as a
  continuous bundle over~\(X\) whose fibres are Morita equivalences.
  Such bundles may be tensored together over~\(X\).  With the graded
  tensor product over~\(X\), we get
  \(\Cliff_{V_1 \oplus V_2} \cong \Cliff_{V_1} \hot \Cliff_{V_2}\).
  Therefore, \(\KR\)\nb-orientations for \(V_1\) and~\(V_2\) induce one for
  \(V_1 \oplus V_2\).  Assume, conversely, that \(V_1 \oplus V_2\)
  and~\(V_2\) are \(\KR\)\nb-oriented.  Then the bundle \(\Cliff_{V_2}\) is
  Morita equivalent to the trivial bundle with
  fibre~\(\Cliff_{a,b}\) for some \(a,b\in\N\).  So
  \(\Cliff_{V_2} \hot \Cliff_{b,a}\) is Morita equivalent to a
  trivial bundle of matrix algebras, making it Morita equivalent to
  the trivial rank-1 bundle \(X\times \C\).  Therefore,
  \(\Cliff_{V_1}\) is Morita equivalent to
  \(\Cliff_{V_1} \hot \Cliff_{V_2} \hot \Cliff_{b,a} \cong
  \Cliff_{V_1 \oplus V_2} \hot \Cliff_{b,a}\).  Since
  \(V_1 \oplus V_2\) is \(\KR\)\nb-oriented, this is also Morita equivalent
  to some trivial Clifford algebra bundle.  This gives a
  \(\KR\)\nb-orientation on~\(V_1\).
\end{proof}

The following proposition clarifies how many \(\KR\)\nb-orientations
a \(\KR\)\nb-orientable vector bundle admits.

\begin{proposition}
  \label{pro:KR-orientation_unique}
  Two \(\KR\)\nb-orientations for the same \(\Z/2\)\nb-equivariant
  vector bundle become isomorphic after tensoring one of them with a
  \(\Z/2\)\nb-graded ``real'' complex line bundle \(L\prto X\),
  which is determined uniquely.
\end{proposition}

\begin{proof}
  Let \(M_1\) and~\(M_2\) be two \(\Cont_0(X)\)-linear ``real''
  graded Morita equivalences between \(\Cont_0(X,\Cliff_E)\) and
  \(\Cont_0(X)\otimes \Cliff_{a,b}\).  Then the composite of \(M_1\)
  and the inverse of~\(M_2\) is a \(\Cont_0(X)\)-linear ``real''
  graded Morita self-equivalence of
  \(\Cont_0(X)\otimes \Cliff_{a,b}\).  Taking the exterior product
  with \(\Cliff_{a,b}\) is clearly bijective on isomorphism classes
  of Morita self-equivalences if \(a=b\) because then
  \(\Cliff_{a,b}\) is a matrix algebra.  If \(a\neq b\), then it is
  also bijective on isomorphism classes because when we tensor first
  with \(\Cliff_{a,b}\) and then with \(\Cliff_{b,a}\), we tensor
  with \(\Cliff_{a+b,a+b}\), which is already known to be bijective
  on isomorphism classes.  Therefore, we may replace the
  self-equivalence of \(\Cont_0(X)\otimes \Cliff_{a,b}\) by one of
  \(\Cont_0(X)\).  Such a Morita self-equivalence of \(\Cont_0(X)\)
  is well known to be just a complex line bundle~\(L\).  In our
  case, the line bundle must also carry a ``real'' involution and a
  \(\Z/2\)-grading, of course.  Unravelling the bijections on
  isomorphism classes, we see that~\(M_2\) is isomorphic to
  \(L \hot M_1\).
\end{proof}

\begin{remark}
  \label{rem:orientation_reversal}
  Assume that~\(X\) is connected.  Let \(L\prto X\) be a ``real''
  complex line bundle.  Then there are two ways to define a
  \(\Z/2\)\nb-grading on~\(L\): we may declare all of~\(L\) to have
  parity \(0\) or~\(1\).  When we tensor a Morita equivalence~\(M\)
  with a line bundle of negative parity, we flip the
  \(\Z/2\)-grading on~\(M\).  Clearly, we may always flip the
  \(\Z/2\)-grading on a Morita equivalence and get another Morita
  equivalence.  This operation is called
  \emph{orientation-reversal}.
\end{remark}

\section{Representable \texorpdfstring{$\KR$\nb-}{KR-}theory and \texorpdfstring{$\KR$\nb-}{KR-}theory with supports}
\label{sec:representable}

We are going to construct a geometric bivariant \(\KR\)\nb-theory for spaces
with a ``real'' involution.  We follow
\cites{Emerson-Meyer:Normal_maps, Emerson-Meyer:Correspondences},
where equivariant geometric bivariant \(\K\)\nb-theory is developed for
spaces with a groupoid action.  While it is mentioned there that the
theory also works for KO\nb-theory, the more general \(\KR\)\nb-theory is
not mentioned there explicitly.  Nevertheless, the theory developed
there applies because \(\KR\)\nb-theory comes from a \(\Z/2\)-equivariant
cohomology theory.  This is checked in some detail
in~\cite{Joseph:Master}.  The proof is similar to the proof of the
same result for \(\Z/2\)-equivariant KO-theory, which only differs
in that the group~\(\Z/2\) acts linearly instead of
conjugate-linearly.  Therefore, we do not repeat the proof here.

We clarify, however, how to define the relevant cohomology theory
because we will use this later anyway.  \(\KR\)\nb-theory, like ordinary
\(\K\)\nb-theory, is not a cohomology theory because it is only
functorial for proper continuous maps.  The cohomology theory from
which \(\K\)\nb-theory comes is called \emph{representable}
\(\K\)\nb-theory.
Adapting the approach in~\cite{Emerson-Meyer:Equivariant_K} for
groupoid-equivariant \(\K\)\nb-theory to the ``real'' case, we
define the representable analogue of \(\KR\)\nb-theory as
\begin{equation}
  \label{eq:def_representable_KR}
  \KR^n_X(X) \defeq \KKR_{-n}^X\bigl(\Cont_0(X),\Cont_0(X)\bigr);
\end{equation}
the right hand side means the \(\Cont_0(X)\)-linear ``real'' version
of Kasparov theory, which Kasparov~\cite{Kasparov:Novikov} denotes
by \(\mathcal{R}\KK\).

The following result is shown in~\cite{Joseph:Master} and is what is
needed to apply the machinery developed in
\cites{Emerson-Meyer:Normal_maps, Emerson-Meyer:Correspondences} to
\(\KR\)\nb-theory:

\begin{theorem}[\cite{Joseph:Master}]
  Representable \(\KR\)\nb-theory is a multiplicative
  \(\Z/2\)-equivariant cohomology theory, and \(\KR\)\nb-oriented
  vector bundles are oriented for it.
\end{theorem}

Our notation for representable \(\KR\)\nb-theory alludes to a further
generalisation, namely, the \(\KR\)\nb-theory \(\KR^*_Z(X)\) of a
space~\(X\) with \(Z\)\nb-compact support, given a map
\(b\colon X\to Z\).  By analogy
to~\cite{Emerson-Meyer:Equivariant_K}, we define this by
\begin{equation}
  \label{eq:def_KR_support}
  \KR^n_Z(X) \defeq \KKR_{-n}^Z\bigl(\Cont_0(Z),\Cont_0(X)\bigr).
\end{equation}
If \(Z=\pt\), then \(\KR_Z=\KR\) is the \(\KR\)\nb-theory as defined
above.  Representable \(\KR\)\nb-theory is the special case where
\(Z=X\) and~\(b\) is the identity map.  More generally, if~\(b\) is
proper, then \(\KR_Z^*(X)=\KR_X^*(X)\) is the representable
\(\KR\)\nb-theory of~\(X\).  In particular, if~\(X\) is compact,
then \(\KR_Z^*(X) = \KR^*(X)\) for all \(b\colon X\to Z\).

Let~\(X\) be a \(\Z/2\)-manifold.  Then any ``real'' complex vector
bundle over~\(X\) defines a class in \(\KR^0_X(X)\).  With our
Kasparov theory definition, this is the Hilbert
\(\Cont_0(X)\)-module of sections of~\(X\) with \(\Cont_0(X)\)
acting by pointwise multiplication also on the left.  This only
defines a class in the usual \(\KR^0(X)\) if~\(X\) is compact.
If~\(X\) is compact, we also know that \(\KR^0(X) = \KR^0_X(X)\) is
the Grothendieck group of the monoid of such ``real'' vector
bundles.

For a finite-dimensional CW-complex~\(X\), it is well known that its
representable \(\K\)\nb-theory and \(\KO\)\nb-theory are the
Grothendieck groups of the monoids of complex and real vector
bundles over~\(X\).  The analogous result for
\emph{\(\Z/2\)-equivariant} \(\K\)\nb-theory is false, however, as
shown by the counterexample in
\cite{Lueck-Oliver:Completion}*{Example~3.11}.  This counterexample
for~\(\K^{\Z/2}\), however, does not work like this for
\(\KR\)\nb-theory.  Therefore, it is possible that the representable
\(\KR\)\nb-theory of a finite-dimensional \(\Z/2\)-CW-complex is
always isomorphic to the Grothendieck group of the monoid of
``real'' complex vector bundles over~\(X\).  We have not
investigated this question.

More generally, let us add a \(\Z/2\)-map \(b\colon X\to Z\).
Choose two ``real'' complex vector bundles~\(E_\pm\) over~\(X\)
together with an isomorphism
\(\varphi\colon E_+|_{X\setminus A} \congto E_-|_{X\setminus A}\)
for an open subset \(A\subseteq X\) whose closure is
\(Z\)\nb-compact in the sense that \(b|_{\cl{A}}\colon \cl{A} \to Z\)
is proper.  We may, of course, equip~\(E_\pm\) with inner products.
We may arrange these so that~\(\varphi\) is unitary.  Using the
Tietze Extension Theorem, we may extend~\(\varphi\) to a continuous
section~\(\tilde\varphi\) of norm~\(1\) of the vector bundle
\(\Hom(E_+,E_-)\) on all of~\(X\).  Then we define a cycle for
\(\KKR_0^Z\bigl(\Cont_0(Z),\Cont_0(X)\bigr)\) as follows.  The
Hilbert module consists of the space of sections of
\(E_+ \oplus E_-\) with the \(\Z/2\)-grading induced by this
decomposition and with the ``real'' structure induced by the
``real'' structures on~\(E_\pm\).  A function \(h\in\Cont_0(Z)\)
acts on this by pointwise multiplication with \(h\circ b\).  The
Fredholm operator is pointwise multiplication
with~\(\tilde\varphi\).  This is indeed a cycle for
\(\KKR_0^Z\bigl(\Cont_0(Z),\Cont_0(X)\bigr)\)
because~\(\tilde\varphi\) is unitary outside a \(Z\)\nb-compact
subset.

With the definition in~\eqref{eq:def_KR_support}, it becomes obvious
that a class \(\xi\in\KR^n_Z(X)\) yields an element of
\([\xi]\in \KKR_{-n}\bigl(\Cont_0(Z),\Cont_0(X)\bigr)\), which
induces maps \(\xi_*\colon \KR^a(Z) \to \KR^{a+n}(X)\) for
\(a\in\Z/8\).  This generalises both the pull back functoriality for
continuous proper maps \(b\colon X\to Z\) and the map on
\(\KR^*(X)\) that multiplies with a vector bundle on~\(X\).

\section{Wrong-Way functoriality of \texorpdfstring{$\KR$\nb-}{KR-}theory}
\label{sec:wrong-way}

In the following, we shall specialise some of the theory in
\cites{Emerson-Meyer:Normal_maps, Emerson-Meyer:Correspondences} and
simplify it a bit for our more limited purposes.  First, we take the
groupoid denoted~\(\mathcal{G}\) in
\cites{Emerson-Meyer:Normal_maps, Emerson-Meyer:Correspondences} to
be the group~\(\Z/2\).  This is because a ``real'' structure is the
same as a \(\Z/2\)\nb-action.  So the object space
of~\(\mathcal{G}\), which is denoted~\(Z\) in
\cites{Emerson-Meyer:Normal_maps, Emerson-Meyer:Correspondences}, is
just the one-point space~\(\pt\).  Thus fibre products over~\(Z\)
become ordinary products.

Secondly, we work in the smooth setting, that is, with smooth
manifolds without boundary, and with smooth maps only.  We also
assume smooth manifolds to be finite-dimensional, which is automatic
if they are connected.  We briefly call a finite-dimensional smooth
manifold with a smooth \(\Z/2\)-action a \emph{\(\Z/2\)-manifold}.

\begin{remark}
  Swan's Theorem says that any vector bundle over a paracompact
  space of finite covering dimension is a direct summand in a
  trivial vector bundle.  This basic result may fail for
  groupoid-equivariant \(\K\)\nb-theory, even when the spaces are
  compact and the groupoid is a bundle of Lie groups (see
  \cite{Emerson-Meyer:Normal_maps}*{Example~2.7}).  This creates the
  need to speak of ``subtrivial'' equivariant vector bundles in the
  general setting considered in \cites{Emerson-Meyer:Normal_maps,
    Emerson-Meyer:Correspondences}.  However, for the finite
  group~\(\Z/2\), any \(\Z/2\)-equivariant vector bundle over a
  \(\Z/2\)-manifold is subtrivial by
  \cite{Emerson-Meyer:Normal_maps}*{Theorem~3.11}.  We use this
  occasion to point out that the hypotheses in that theorem are
  wrong: it should be assumed that the space~\(Y\) and not~\(X\) is
  finite-dimensional.  Any smooth finite-dimensional manifold~\(X\)
  with smooth \(\Z/2\)-action has a structure of finite-dimensional
  \(\Z/2\)-CW-complex.  In the following, we may therefore drop the
  adjective ``subtrivial'' as long as we restrict attention to
  \(\Z/2\)-manifolds.
\end{remark}

\begin{remark}
  Since~\(\Z/2\) is a finite group, its regular representation is a
  ``full vector bundle'' over~\(\pt\).  Such a vector bundle is
  needed for several results in \cites{Emerson-Meyer:Normal_maps,
    Emerson-Meyer:Correspondences}, and it comes for free in our
  case.  Any linear representation of~\(\Z/2\) is isomorphic to a
  direct sum of copies of the two characters of~\(\Z/2\).  That is,
  it is isomorphic to~\(\R^{a,b}\) where the generator of~\(\Z/2\)
  acts by the ``real'' involution on that space.
\end{remark}

After these preliminary remarks, we construct wrong-way
functoriality or shriek maps for \(\KR\)\nb-oriented smooth maps.
This is based on factorising smooth maps in a certain way.  The
factorisation is called a \emph{normally nonsingular map} in
\cites{Emerson-Meyer:Normal_maps, Emerson-Meyer:Correspondences}.
Under the extra assumptions that we impose, any smooth map has such
a factorisation and it is unique up to ``smooth equivalence''; this
implies that the shriek map is independent of the factorisation.  So
the factorisation becomes irrelevant in the special case that we
consider here.  This fails already for smooth maps between smooth
manifolds with boundary (see
\cite{Emerson-Meyer:Normal_maps}*{Example~4.7}).  So the theory of
normally nonsingular maps is needed to treat this more general class
of spaces.  Our applications, however, concern only manifolds
without boundary.  Therefore, we will define \(\KR\)\nb-oriented
correspondences only in this special case to simplify the theory.
Nevertheless, we include the basic definition of a normally
nonsingular map to clarify what would be needed to extend the theory
to more general spaces than \(\Z/2\)-manifolds.

\begin{definition}
  \label{def:2.3}
  Let \(X\) and~\(Y\) be \(\Z/2\)-manifolds.  A (smooth) \emph{normally
    nonsingular \(\Z/2\)-map} from \(X\) to~\(Y\) consists of the
  following data:
  \begin{itemize}
  \item \(V\), a \(\Z/2\)-equivariant \(\R\)\nb-vector bundle over~\(X\);
  \item \(E=\R^{a,b}\) for some \(a,b\in\N\), an \(\R\)\nb-linear
    representation of the group~\(\Z/2\), which we treat as a
    \(\Z/2\)-equivariant vector bundle over~\(\pt\);
  \item \(\hat{f}\colon \abs{V} \hookrightarrow Y\times \R^{a,b}\), a
    \(\Z/2\)-equivariant diffeomorphism between the total space
    of~\(V\) and an open subset of
    \(Y\times \R^{a,b} = Y \times_\pt \abs{E}\).
  \end{itemize}
  Here~\(\abs{V}\) is the total space of the vector bundle~\(V\).
  Let \(\zeta_V\colon X\to \abs{V}\) be the zero section of~\(V\)
  and let \(\pi_Y\colon Y\times \R^{a,b} \to Y\) be the coordinate
  projection.  The \emph{trace} of \((V,E,\hat{f})\) is the dotted
  composite map
  \begin{equation}
    \label{eq:nns_maps}
    \begin{tikzcd}
      X \arrow[r, dotted, rightarrow, "f"]
      \arrow[d, rightarrow,"\zeta_V"] &
      Y \\
      \abs{V} \arrow[r, hookrightarrow, "\hat{f}"] &
      Y\times \R^{a,b}. \arrow[u, rightarrow,"\pi_Y"]
    \end{tikzcd}
  \end{equation}
\end{definition}

The following definition describes two ways to change a normally
nonsingular map.  The important feature is that they do not change
the resulting shriek map.

\begin{definition}
  Let \((V,E,\hat{f})\) be a normally nonsingular map and let
  \(E_0\) be another linear representation of~\(\Z/2\).  Then
  \((V\oplus (X\times E_0),E\oplus E_0, \hat{f} \times \Id_{E_0})\)
  is another normally nonsingular map, called a \emph{lifting} of
  \((V,E,\hat{f})\).  A normally nonsingular map from
  \(X\times [0,1]\) to \(Y\times [0,1]\) where the map~\(\hat{f}\)
  is a map over~\([0,1]\) is called an \emph{isotopy} between the
  two normally nonsingular maps that arise by restricting to the end
  points of \([0,1]\).  Two smooth normally nonsingular maps are
  called \emph{smoothly equivalent} if they have liftings that are
  isotopic.
\end{definition}

\begin{proposition}[\cite{Emerson-Meyer:Normal_maps}*{Theorems 3.25
    and~4.36}]
  \label{pro:exmp1}
  Any smooth \(\Z/2\)-map is the trace of a normally nonsingular
  map, which is unique up to smooth equivalence.  Two smooth
  \(\Z/2\)-maps are smoothly homotopic if and only if their lifts to
  normally nonsingular maps are smoothly equivalent.
\end{proposition}

We recall how to lift a smooth \(\Z/2\)-map \(f\colon X\to Y\) to a
normally nonsingular map.  There is a \(\Z/2\)-equivariant, proper
embedding \(i\colon X\to\R^{a,b}\) for some \(a,b\in\N\).  Let
\(E\defeq \R^{a,b}\).  The map
\((f,i)\colon X \to Y\times \R^{a,b}\) is still a
\(\Z/2\)-equivariant, proper embedding.  Let~\(V\) be its normal
bundle.  By the Tubular Neighbourhood Theorem, the total space
of~\(V\) is \(\Z/2\)-equivariantly diffeomorphic to a neighbourhood
of the image of~\((f,i)\).  This gives the open
embedding~\(\hat{f}\) for a normally nonsingular map.

\begin{definition}
  \label{def:KR-orientation_map}
  Let \(f\colon X\to Y\) be a smooth \(\Z/2\)-map.  A \emph{stable
    normal bundle} for~\(f\) is a vector bundle \(V \prto X\) such
  that \(V \oplus T X\) is isomorphic to
  \(f^*(TY) \oplus (X\times \R^{a,b})\) for some \(a,b\in\N\).  A
  \emph{\(\KR\)\nb-orientation} for~\(f\) is a stable normal
  bundle~\(V\) for~\(f\) together with a \(\KR\)\nb-orientation
  of~\(V\).  Its \emph{\(\KR\)\nb-dimension} is the difference of
  the \(\KR\)\nb-dimension of~\(V\) and \(a-b \bmod 8\), the
  \(\KR\)\nb-dimension of the trivial bundle
  \(X\times \R^{a,b} \prto X\).
\end{definition}

\begin{lemma}
  \label{lem:KR-orientation}
  Let \((V,E,\hat{f})\) be a normally nonsingular \(\Z/2\)-map and
  let \(f\colon X\to Y\) be its trace.  Then a \(\KR\)\nb-orientation
  for~\(f\) is equivalent to a \(\KR\)\nb-orientation of the vector
  bundle~\(V\).
\end{lemma}

\begin{proof}
  The bundle~\(V\) is a stable normal bundle for~\(f\).  Therefore,
  a \(\KR\)\nb-orientation of~\(V\) induces one for~\(f\).  Any two
  stable normal bundles of~\(f\) are stably isomorphic.  Thus
  Lemma~\ref{lem:add_KR-orientations} implies that a
  \(\KR\)\nb-orientation for one stable normal bundle of~\(f\)
  induces \(\KR\)\nb-orientations on all other stable normal bundles
  of~\(f\).
\end{proof}

\begin{definition}
  A \emph{\(\KR\)\nb-orientation} on a \(\Z/2\)-manifold~\(X\) is a
  \(\KR\)\nb-orientation on its tangent vector bundle \(T X\).  The
  \emph{\(\KR\)\nb-dimension} \(\dim_\KR X\) is defined as
  \(\dim_\KR T X\).
\end{definition}

If \(X\) and~\(Y\) are two \(\KR\)\nb-oriented \(\Z/2\)-manifolds, then any
smooth map \(f\colon X\to Y\) inherits a \(\KR\)\nb-orientation by
Lemma~\ref{lem:add_KR-orientations}, and its \(\KR\)\nb-dimension is
\begin{equation}
  \label{eq:KR-dim_for_map}
  \dim_\KR f = \dim_\KR Y - \dim_\KR X.  
\end{equation}

\begin{example}
  \label{exa:KR-oriented_manifolds}
  We are going to describe \(\KR\)\nb-orientations on the
  \(\Z/2\)-manifolds \(\R^{a,b}\), \(\Sphere^{a,b}\), and~\(\T^d\).
  The tangent bundle of~\(\R^{a,b}\) is the trivial bundle with
  fibre~\(\R^{a,b}\).  So its Clifford algebra bundle is already
  isomorphic to the trivial Clifford algebra bundle with fibre
  \(\Cliff_{a,b}\).

  The covering \(\R\to\T\), \(t\mapsto \exp(\ima t)\), becomes a
  ``real'' covering \(\R^{0,1} \to \T\).  This induces an
  isomorphism between the tangent bundle of~\(\T\) and the trivial
  bundle with fibre~\(\R^{0,1}\).  Taking a \(d\)\nb-fold product,
  we get an isomorphism from the tangent bundle of~\(\T^d\) to the
  trivial bundle with fibre~\(\R^{0,d}\).  This induces an
  isomorphism between the Clifford algebra bundle of \(T\T^d\) and
  the trivial bundle \(\T^d\times\Cliff_{0,d}\).

  The outward pointing radial vector field \(\partial/\partial r\)
  on the unit sphere spans the normal bundle of the inclusion map
  \(\Sphere^{a,b} \to \R^{a,b}\).  Since the involution
  on~\(\R^{a,b}\) is linear, it preserves this vector field, that
  is, the normal bundle is \(\Z/2\)-equivariantly isomorphic to the
  trivial bundle with fibre~\(\R^{1,0}\).  So
  \(T \Sphere^{a,b} \oplus (\Sphere^{a,b}\times \R^{1,0}) \cong
  \Sphere^{a,b} \times \R^{a,b}\).  Then
  Lemma~\ref{lem:add_KR-orientations} gives a \(\KR\)\nb-orientation
  on~\(T \Sphere^{a,b}\).  We also find
  \[
    \dim_\KR (\R^{a,b}) = a-b,\qquad
    \dim_\KR (\T^d) = -d,\qquad
    \dim_\KR (\Sphere^{a,b}) = a-b-1,
  \]
  where the three numbers are understood as elements of~\(\Z/8\).
  Notice that the above argument for~\(\Sphere^{a,b}\) also works for
  \(a=0\).
\end{example}

Let \((V,E,\hat{f})\) be a \(\KR\)\nb-oriented, normally nonsingular
\(\Z/2\)-map and let \(f\colon X\to Y\) be its trace.  Then there
are Thom isomorphisms both for the vector bundle \(V \prto X\) and
the trivial vector bundle \(Y\times\R^{a,b}\prto Y\).  The open
inclusion~\(\hat{f}\) identifies \(\Cont_0(\abs{V})\) with an ideal
in \(\Cont_0(Y\times\R^{a,b})\).  Thus each of the solid maps
in~\eqref{eq:nns_maps} induces a map in \(\KR\)\nb-theory.  Taking the
composite gives induced maps
\[
  f_!\colon \KR^n(X) \to \KR^{n+\dim_\KR f}(Y)
  \qquad \text{for }n\in\Z/8.
\]
Since the Thom isomorphisms come from KKR-classes, the map~\(f_!\)
is the Kasparov product with an element in
\(\KKR_{-\dim_\KR f}(\Cont_0(X),\Cont_0(Y))\), which we will also
denote by~\(f_!\).  It turns out that~\(f_!\) is preserved under
lifting and isotopy and is
functorial in the sense that
\[
  f_! \circ g_! = (f\circ g)_!
\]
for two composable \(\Z/2\)-maps \(f,g\)
(see~\cite{Emerson-Meyer:Normal_maps}).  In particular, by
Proposition~\ref{pro:exmp1}, \(f_!\) only depends on the smooth
map~\(f\) and its \(\KR\)\nb-orientation, justifying the notation~\(f_!\).
The map~\(f_!\) is called the \emph{shriek map} of~\(f\), and the
map that sends~\(f\) to~\(f_!\) is the wrong-way functoriality of
\(\KR\)\nb-theory.

\begin{example}
  Let \(E\prto X\) be a \(\KR\)\nb-oriented vector bundle.  Then the zero
  section \(\zeta_E\colon X\hookrightarrow E\) and the bundle
  projection \(\pi_E\colon E\prto X\) are \(\KR\)\nb-oriented in a canonical
  way, such that their shriek maps are the Thom isomorphism and its
  inverse \(\KR^n(X) \leftrightarrow \KR^{n+\dim_\KR E}(E)\).
\end{example}

\begin{example}
  Let~\(X\) be a \(\Z/2\)-manifold.  Let~\(T X\) be its tangent
  bundle.  As a complex manifold, it carries a canonical
  \(\KR\)\nb-orientation of \(\KR\)\nb-dimension~\(0\) (see
  Example~\ref{exa:complex_oriented}).  This induces a
  \(\KR\)\nb-orientation on the constant map \(f\colon T X \to \pt\)
  to the one-point space.  The shriek map
  \(f_!\colon \KR^0(T X) \to \KR^0(\pt) = \Z\) is the Atiyah--Singer
  topological index map on~\(X\).
\end{example}

\begin{remark}
  If~\(f\) is a smooth submersion, then~\(f_!\) has an
  \emph{analytic} variant~\(f_{!,\an}\), given by the class in
  Kasparov theory of the family of Dirac operators along the fibres
  of~\(f\).  This analytic version is equal to the topological one,
  that is, \(f_{!,\an} = f_!\) holds in
  \(\KKR_{-\dim_\KR f}(\Cont_0(X),\Cont_0(Y))\).  The proof is the
  same as for \cite{Emerson-Meyer:Normal_maps}*{Theorem~6.1}, which
  deals with the analogous statement in \(\KK\), that is, when we
  forget the ``real'' structures.  This is equivalent to the
  families version of the Atiyah--Singer index theorem for the
  family of Dirac operators along the fibres of the
  submersion~\(f\).
\end{remark}

We are particularly interested in the shriek maps for an inclusion
\(f\colon \pt \to X\), where~\(\pt\) denotes the one-point space
with the trivial ``real'' structure.  So \(f(\pt)\in X\) must be a
fixed point of the ``real'' involution on~\(X\).  The map~\(f\) is
\(\KR\)\nb-orientable because any vector bundle over a point is
trivial and thus \(\KR\)\nb-orientable.  A \(\KR\)\nb-orientation
for~\(X\) chooses a canonical \(\KR\)\nb-orientation for~\(f\).  Let
us specialise to the case where \(X= \Sphere^{a,b}\) with \(a>1\),
so that \(S\) and~\(N\) are fixed points.  Let
\(S\colon \pt \to \Sphere^{a,b}\) be the inclusion of the north
pole.  Give~\(\Sphere^{a,b}\) the \(\KR\)\nb-orientation from
Example~\ref{exa:KR-oriented_manifolds}.  Then we get a canonical
map
\[
  S_!\colon \Z
  \cong \KR^0(\pt)
  \to \KR^{a-b-1}(\Sphere^{a,b})
  \cong \DK(\Cont(\Sphere^{a,b}) \otimes \Cliff_{a,b}).
\]

\begin{lemma}
  \label{lem:beta_from_point}
  There is \(\chi \in \{\pm1\}\) such that the map~\(S_!\) sends the
  generator \(1\in\Z\) to
  \(\chi\cdot ([\beta_{a,b}] - [\gamma_1])\).
\end{lemma}

\begin{proof}
  The stereographic projection at the north pole induces a
  \(\Z/2\)-equivariant diffeomorphism
  \(\Sphere^{a,b} \setminus \{N\} \cong \R^{a-1,b}\).  We could use
  this as a tubular neighbourhood for the inclusion of the south
  pole into~\(\Sphere^{a,b}\).  Therefore, \(S_!\) factors through
  an isomorphism onto the direct summand
  \(\KR^{a-b-1}(\Sphere^{a,b}\setminus \{N\})\cong \Z\).
  Since~\([\beta_{a,b}] - [\gamma_1]\) generates this summand,
  \(S_!\) must send the generator \(1\in\Z\) to
  \(\pm ([\beta_{a,b}] - [\gamma_1])\).
\end{proof}

The sign~\(\chi\) depends on the choices of \(\KR\)\nb-orientations
and signs in boundary maps, and we do not compute it.  It will
appear in several formulas below.

The power of geometric bivariant \(\K\)\nb-theory is that there
often is a simple way to compute composite maps \(b_* \circ f_!\)
for a \(\KR\)\nb-oriented map \(f\colon X\to Y\) and a proper
continuous map \(b\colon Z\to Y\).  We will use this to compute the
pullback of \([\beta_{1,d}]\) to the van Daele \(\K\)\nb-theory of~\(\T^d\)
 using only geometric considerations.  This
circumvents rather messy Chern character computations in the physics
literature (see \cites{Golterman-Jansen-Kaplan:Currents,
  Qi-Hughes-Zhang:TFT_time-reversal}), and it also works in the real
case.  The following proposition makes precise when and how we may
compute \(b_* \circ f_!\):

\begin{proposition}
  \label{pro:reorder_pull-back_shriek}
  Let \(X\) and~\(Y\) be \(\Z/2\)-manifolds.  Let \(f\colon X\to Y\)
  be a \(\KR\)\nb-oriented smooth \(\Z/2\)-map and let
  \(b\colon Z\to Y\) be a proper smooth \(\Z/2\)-map.  Assume that
  \(b\) and~\(f\) are transverse.  Then the fibre product
  \(X\times_Y Z\) is a smooth \(\Z/2\)-manifold.  The projection
  \(\pi_X\colon X\times_Y Z \to X\) is proper, and
  \(\pi_Z\colon X\times_Y Z\to Z\) inherits from~\(f\) a
  \(\KR\)\nb-orientation of the same \(\KR\)\nb-dimension.  The
  following diagram commutes:
  \begin{equation}
    \label{eq:shriek_pullback_diagram}
    \begin{tikzcd}
      \KR^n(X) \ar[r, "f_!"] \ar[d, "\pi_X^*"'] &
      \KR^{n+\dim_\KR f}(Y) \ar[d, "b^*"] \\
      \KR^n(X \times_Y Z) \ar[r, "(\pi_Z)_!"] &
      \KR^{n+\dim_\KR f}(Z).
    \end{tikzcd}
  \end{equation}
\end{proposition}

\begin{proof}
  This is a special case of
  \cite{Emerson-Meyer:Correspondences}*{Theorem~2.32} about
  composing \(\KR\)\nb-oriented correspondences.  We will discuss a
  more general result later.  Here we only describe the
  \(\KR\)\nb-orientation that the projection~\(\pi_Z\) inherits.
  The transversality assumption implies that there is an exact
  sequence of vector bundles over \(X\times_Y Z\) as follows:
  \[
    T(X\times_Y Z) \into \pi_X^*(T X) \oplus \pi_Z^*(T Z) \prto
    \pi^*(T Y).
  \]
  Since any such extension splits \(\Z/2\)-equivariantly, it follows
  that
  \[
    \pi^*(T Y) \oplus T(X\times_Y Z)
    \cong \pi_X^*(T X) \oplus \pi_Z^*(T Z).
  \]
  Since~\(f\) is \(\KR\)\nb-oriented, there are \(a,b\in\N\) and a
  \(\KR\)\nb-oriented vector bundle \(V\prto X\) such that
  \(f^*(T Y) \oplus (X\times \R^{a,b}) \cong T X \oplus V\).  Then
  \(\pi^*(T Y) \oplus (X\times_Y Z\times \R^{a,b}) \cong \pi_X^*(T
  X) \oplus \pi_X^*(V)\) and hence
  \[
    \pi_X^*(T X) \oplus \pi_Z^*(T Z)
    \oplus (X\times_Y Z\times \R^{a,b})
    \cong
    \pi_X^*(T X) \oplus \pi_X^*(V) \oplus T(X\times_Y Z).
  \]
  Since the \(\Z/2\)-equivariant vector bundle \(\pi_X^*(T X)\) is
  subtrivial, we may add another \(\Z/2\)-equivariant vector bundle
  to arrive at an isomorphism
  \[
    \pi_Z^*(T Z) \oplus (X\times_Y Z\times \R^{a+a',b+b'})
    \cong (X\times_Y Z\times \R^{a',b'}) \oplus \pi_X^*(V)
    \oplus T(X\times_Y Z).
  \]
  Thus \((X\times_Y Z\times \R^{a',b'}) \oplus \pi_X^*(V)\) is a
  stable normal bundle for~\(\pi_Z\).  It inherits a
  \(\KR\)\nb-orientation from the obvious \(\KR\)\nb-orientation on
  the trivial bundle and the given \(\KR\)\nb-orientation of~\(V\)
  by Lemma~\ref{lem:add_KR-orientations}.
\end{proof}

We will later need the special case when~\(f\) is the inclusion of a
point \(y_0\in Y\), so that \(X=\pt\).  A \(\KR\)\nb-orientation of~\(f\) is
the same as a \(\KR\)\nb-orientation on the tangent space~\(T_{y_0} Y\).  In
this case, the coordinate projection~\(\pi_Z\) is a bijection
between \(X\times_Y Z\) and the preimage \(b^{-1}(y_0)\) of~\(y_0\)
in~\(Z\).  Transversality means in this case that the differential
of~\(b\) is surjective in all points of \(b^{-1}(y_0)\).  Then
\(b^{-1}(y_0)\) is a smooth submanifold.  It is also compact
because~\(b\) is proper.  The normal bundle of the inclusion of
\(b^{-1}(y_0)\) into~\(Z\) is canonically isomorphic to the trivial
bundle with fibre~\(T_{y_0} Y\).  This gives the induced
\(\KR\)\nb-orientation of~\(\pi_Z\).

The commuting diagram~\eqref{eq:shriek_pullback_diagram} is best
understood in the setting of geometric bivariant \(\KR\)\nb-theory.  This
theory describes \(\KKR_*(\Cont_0(X),\Cont_0(Y))\) for two
\(\Z/2\)-manifolds \(X\) and~\(Y\) in a geometric way.  Its main
ingredients are
\[
  b_* \in \KKR_0(\Cont_0(Y),\Cont_0(Z)),\qquad
  f_! \in \KKR_{-\dim_\KR(f)}(\Cont_0(X),\Cont_0(Y))
\]
for proper smooth maps \(b\colon Z\to Y\) and \(\KR\)\nb-oriented
smooth maps \(f\colon X\to Y\).  The details are explained in the
next section.

\section{Geometric bivariant \texorpdfstring{$\KR$\nb-}{KR-}theory}
\label{sec:correspondences}

We now define \(\KR\)-oriented correspondences between two
\(\Z/2\)-manifolds \(X\) and~\(Y\) as
in~\cite{Emerson-Meyer:Correspondences}.  These produce a geometric
version of bivariant ``real'' Kasparov theory.  The definition of a
correspondence in~\cite{Emerson-Meyer:Correspondences} differs
slightly from the original definition of correspondences by Connes
and Skandalis in~\cite{Connes-Skandalis:Longitudinal}*{§III} by
allowing maps~\(b\) that fail to be proper.  This greatly simplifies
the proof that the geometric and analytic bivariant
\(\KR\)\nb-theories agree.  We shall use smooth maps, whereas the
definition in~\cite{Emerson-Meyer:Correspondences} uses normally
nonsingular maps.  This makes no difference for \(\Z/2\)-manifolds
because of Proposition~\ref{pro:exmp1}.

\begin{definition}
  \label{def3.12}
  Let \(X\) and~\(Y\) be \(\Z/2\)-manifolds.  A (smooth)
  \emph{\(\KR\)-oriented correspondence} from~\(X\) to~\(Y\) is a
  quadruple \((M,b,f,\xi)\), where
  \begin{itemize}
  \item \(M\) is a \(\Z/2\)-manifold,
  \item \(b\colon  M\to X\) is a smooth \(\Z/2\)\nb-map,
  \item \(f\colon M \to Y\) is a \(\KR\)-oriented smooth
    \(\Z/2\)\nb-map, and
  \item \(\xi\in \KR^a_X(M)\) for some \(a\in\Z/8\); here
    \(X\)\nb-compact support in~\(M\) refers to the map
    \(b\colon M\to X\).
  \end{itemize}
  The \emph{\(\KR\)\nb-dimension} of the correspondence is defined
  as \(a+ \dim_\KR(f)\).  We often depict a correspondence as
  \[
    \begin{tikzcd}[sep=small]
      & (M,\xi) \ar[dl, "b"'] \ar[dr, "f"]\\
      X && Y.
    \end{tikzcd}
  \]
\end{definition}

The letters \(f\) and~\(b\) in the definition stand for ``forwards''
and ``backwards''.

\begin{example}
  \label{exa:corr_proper_map}
  A proper, smooth \(\Z/2\)-map \(b\colon Y \to X\) yields the following
  correspondence from~\(X\) to~\(Y\):
  \[
    \begin{tikzcd}[sep=small]
      & (Y,1) \ar[dl, "b"'] \ar[dr, equal, "\Id_Y"]\\
      X && Y
    \end{tikzcd}
  \]
  Here \(\Id_Y\) is the identity map with its canonical
  \(\KR\)-orientation, which is indeed a unit arrow in the category
  of \(\KR\)-oriented smooth maps, and \(1 \in \KR^0_X(Y)\) comes from
  the trivial ``real'' vector bundle of rank~\(1\).
\end{example}

\begin{example}
  \label{exa:corr_K-oriented_map}
  A \(\KR\)-oriented smooth \(\Z/2\)-map \(f \colon X \to Y\) yields a
  \(\KR\)-oriented correspondence from~\(X\) to~\(Y\):
  \[
    \begin{tikzcd}[sep=small]
      & (X,1) \ar[dl, equal, "\Id_X"'] \ar[dr, "f"]\\
      X && Y
    \end{tikzcd}
  \]
\end{example}

\begin{example}
  \label{exa:corr_K-class}
  Any class~\(\xi\) in the representable \(\KR\)\nb-theory
  \(\KR_X^*(X)\) of~\(X\) yields a correspondence from~\(X\) to
  itself:
  \[
    \begin{tikzcd}[sep=small]
      & (X,\xi) \ar[dl, equal, "\Id_X"'] \ar[dr, equal, "\Id_X"]\\
      X && X
    \end{tikzcd}
  \]
\end{example}

We shall mainly consider correspondences where the map~\(b\) is
proper and~\(\xi\) is the unit element in \(\KR^0_X(X)\) as in
Example~\ref{exa:corr_proper_map}.

A \(\KR\)\nb-oriented correspondence \((M,b,f,\xi)\) induces an
element
\[
  f_! \circ [\xi] \in \KKR_{-a-\dim_\KR(f)}
  \bigl(\Cont_0(X),\Cont_0(Y)\bigr),
\]
which induces maps
\[
  f_! \circ \xi_*\colon \KR^p(X) \to \KR^{p+a+\dim_\KR(f)}(Y)
  \qquad \text{for all }p\in\Z/8.
\]
For instance, the correspondence in
Example~\ref{exa:corr_proper_map} gives the class of the
\Star{}\alb{}homomorphism \(b^*\colon \Cont_0(X) \to \Cont_0(Y)\)
in~\eqref{eq:def_bstar} in
\(\KKR_0\bigl(\Cont_0(X),\Cont_0(Y)\bigr)\); the correspondence in
Example~\ref{exa:corr_K-oriented_map} gives~\(f_!\);
the correspondence in Example~\ref{exa:corr_K-class} gives the image
of~\(\xi\) in \(\KKR_0\bigl(\Cont_0(X),\Cont_0(X)\bigr)\) under the
forgetful map.

The \emph{geometric bivariant \(\KR\)\nb-theory} \(\widehat{\KR}{}^*(X,Y)\) is
defined as the set of equivalence classes of \(\KR\)\nb-oriented
correspondences, where ``equivalence'' means the equivalence
relation generated by \emph{bordism} (see
\cite{Emerson-Meyer:Correspondences}*{Definition~2.7}) and
\emph{Thom modification} (see
\cite{Emerson-Meyer:Correspondences}*{Definition~2.8}), which
replaces~\(M\) in a correspondence by the total space of a
\(\KR\)-oriented vector bundle over it.  We shall not use the
precise form of these relations below and therefore do not repeat
them here.  It is shown in~\cite{Emerson-Meyer:Correspondences} that
the disjoint union of correspondences makes \(\widehat{\KR}{}^*(X,Y)\)
an Abelian group.

\begin{theorem}
  Let \(n\in\Z/8\) and let \(X\) and~\(Y\) be \(\Z/2\)-manifolds.
  The map above from \(\KR\)-oriented correspondences to Kasparov
  theory defines an isomorphism
  \[
    \widehat{\KR}{}^n(X,Y)
    \cong \KKR_{-n}\bigl(\Cont_0(X),\Cont_0(Y)\bigr).
  \]
\end{theorem}

\begin{proof}
  This is shown by following the proof
  in~\cite{Emerson-Meyer:Correspondences} for equivariant KK-theory.
  The same arguments as in the proof of
  \cite{Emerson-Meyer:Correspondences}*{Theorem~4.2} show that the
  map is well defined and a functor for the composition of
  \(\KR\)-oriented correspondences defined
  in~\cite{Emerson-Meyer:Correspondences}.  Next,
  \cite{Emerson-Meyer:Correspondences}*{Theorem~2.25} for the
  \(\Z/2\)-equivariant cohomology theory \(\KR\) shows that the map
  in the theorem is bijective if \(X=\pt\).  The bivariant case is
  reduced to this easy case using Poincaré duality.  Any
  \(\Z/2\)-manifold~\(X\) admits a ``symmetric dual'' in \(\KR\)\nb-theory
  because of \cite{Emerson-Meyer:Correspondences}*{Theorem~3.17}.
  This is another \(\Z/2\)-manifold~\(P\) with a smooth \(\Z/2\)-map
  \(P\to X\) such that there are duality isomorphisms
  \(\widehat{\KR}{}^n(X,Y) \cong \KR_X^{n+a}(Y \times P)\) in
  geometric bivariant \(\KR\)\nb-theory.  As in
  \cite{Emerson-Meyer:Correspondences}*{Theorem~4.2}, the geometric
  bivariant \(\KR\)-theory classes that give this duality also give
  an isomorphism
  \(\KKR_{-n}\bigl(\Cont_0(X),\Cont_0(Y)\bigr) \cong \KR_X^{n+a}(Y
  \times P)\).
\end{proof}

Thus \(\KR\)\nb-oriented correspondences provide a purely geometric
way to describe Kasparov cycles between the \(\Cst\)\nb-algebras of
functions on \(\Z/2\)-manifolds.  An important feature is that the
Kasparov product may be computed geometrically under an extra
transversality assumption:

\begin{theorem}
  \label{the:compose_corr}
  Let two composable smooth \(\KR\)-oriented correspondences be
  given, as described by the solid arrows in the following diagram:
  
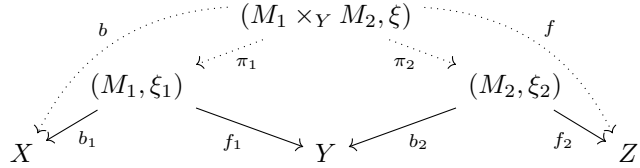
\[
    \begin{tikzcd}[sep=small]
      &&(M_1 \times_Y M_2,\xi) \ar[dl, dotted, "\pi_1"]
      \ar[dr, dotted, "\pi_2"']
      \ar[ddll, dotted, bend right, "b"']
      \ar[ddrr, dotted, bend left, "f"]
      \\
      & (M_1,\xi_1)  \ar[dl, "b_1"]  \ar[dr, "f_1"'] &&
      (M_2,\xi_2)  \ar[dl, "b_2"]  \ar[dr, "f_2"'] \\
      X&&Y&&Z
    \end{tikzcd}
  \]
  Assume that the smooth maps \(f_1\) and~\(b_2\) are transverse,
  that is, if \(m_1 \in M_1\), \(m_2 \in M_2\) satisfy
  \(y \defeq f_1(m_1) = b_2(m_2)\), then
  \(D f_1(T_{m_1} M_1) + D b_2(T_{m_2} M_2) = T_y Y\).
  Then the following hold.  First, \(M_1 \times_Y M_2\) is a smooth
  \(\Z/2\)-manifold.  Secondly, the exterior tensor product
  \(\xi \defeq \pi_1^*(\xi_1) \otimes_Y \pi_2^*(\xi_2)\) has
  \(X\)\nb-compact support with respect to
  \(b \defeq b_1 \circ \pi_1\), that is, it belongs to
  \(\KR^*_X(M_1 \times_Y M_2)\).  Thirdly, the composite map
  \(f \defeq f_2 \circ \pi_2\) inherits a \(\KR\)\nb-orientation
  from \(f_1\) and~\(f_2\), whose \(\KR\)-dimension is
  \(\dim_\KR f = \dim_\KR f_1 + \dim_\KR f_2\).  Finally, the
  resulting \(\KR\)-oriented correspondence
  \((M_1\times_Y M_2,b,f,\xi)\) is the composite of the two given
  \(\KR\)-oriented correspondences, that is, its image in Kasparov
  theory is the Kasparov product of the Kasparov theory images of
  \((M_1,b_1,f_1,\xi_1)\) and \((M_2,b_2,f_2,\xi_2)\).
\end{theorem}

\begin{proof}
  In~\cite{Emerson-Meyer:Correspondences}, the composition of
  correspondences is first defined in a special case.  Any
  \(\KR\)-oriented correspondence is equivalent to one where the
  forward map is the restriction of the coordinate projection
  \(Y\times \R^{a,b} \prto Y\) to an open subset
  \(M\subseteq Y\times \R^{a,b}\) (see
  \cite{Emerson-Meyer:Correspondences}*{Theorem~2.24}).  Then~\(f\)
  is a submersion and hence transverse to any smooth map.  The
  composition of ``special'' correspondences gives geometric
  bivariant \(\KR\)-theory a category structure, and the canonical
  map to Kasparov theory is a functor.  With this preparation, the
  claim in our theorem mostly follows from
  \cite{Emerson-Meyer:Correspondences}*{Theorem~2.32 and
    Example~2.31}.  That theorem says that the ``intersection
  product'', which is described in the statement of the theorem, is
  equivalent to the composite in geometric bivariant
  \(\KR\)\nb-theory.  The example
  in~\cite{Emerson-Meyer:Correspondences} says that the usual
  transversality notion from differential geometry implies the
  transversality assumption that is assumed in the theorem
  in~\cite{Emerson-Meyer:Correspondences} (which makes sense for
  correspondences with a normally nonsingular forward map).
\end{proof}

Proposition~\ref{pro:reorder_pull-back_shriek} is the special case
of Theorem~\ref{the:compose_corr} where \(b_1\) and~\(f_2\) are
identity maps, \(b_2\) is proper, and \(\xi_1\) and~\(\xi_2\) are
the units in representable \(\KR\)\nb-theory.

\begin{remark}
  \label{rem:exterior}
  The \emph{exterior product} of two \(\KR\)-oriented
  correspondences is defined by simply taking the product of all
  spaces and maps and the exterior product of the \(\KR\)-theory
  classes involved.  This defines a symmetric monoidal structure on
  geometric bivariant \(\KR\)\nb-theory by
  \cite{Emerson-Meyer:Correspondences}*{Theorem~2.27}.  It is easy
  to check that it lifts the exterior product on Kasparov theory
  (compare \cite{Emerson-Meyer:Correspondences}*{Theorem~4.2}).
\end{remark}

\section{Model Hamiltonians for the strong topological phase}
\label{sec:Hamiltonians}

In this section we compute the topological phase of certain
translation-invariant Hamiltonians that have been considered in the
mathematical physics literature.  They are given by the formula
\[
  H_m \defeq \frac{1}{2\ima} \sum_{j=1}^d (S_j - S_j^*) \otimes \gamma_j
  + \Bigl(m + \frac{1}{2} \sum_{j=1}^d (S_j + S_j^*)\Bigr)
  \otimes \gamma_0
  \in\Cst(\Z^d)\otimes\Cliff_{1,d}
\]
for \(m\in\R\) (see
\cite{Prodan-Schulz-Baldes:Bulk_boundary}*{§2.2.4 and §2.3.3}).
Here \(S_i \in \Cst(\Z^d)\) is the unitary for the element
\(e_i\in\Z^d\).  Throughout this section, we start numbering
Clifford generators for \(\Cliff_{1,d}\) at~\(0\) and not at~\(1\)
as before.  We give \(\Cst(\Z^d)\) the trivial \(\Z/2\)-grading and
the ``real'' involution where the generators~\(S_i\) are real.  Then
the element~\(H_m\) is selfadjoint, odd, and real because
\(\gamma_0, \ima \gamma_1,\dotsc,\ima \gamma_d\) are real by our
conventions.  The Fourier transform identifies the commutative
\(\Cst\)\nb-algebra \(\Cst(\Z^d)\) with the \(\Cst\)\nb-algebra of
continuous functions on the torus~\(\T^d\).  Our ``real'' structure
on \(\Cst(\Z^d)\) transforms under this isomorphism to the ``real''
structure on~\(\T^d\) in~\eqref{eq:Torus_d_detail}.

Define
\(\tilde\beta_{1,d}\colon \R^{1,d} \to \Cliff_{1,d}\) as
in~\eqref{eq:tildebeta} and define
\[
  \tilde{\varphi}_m\colon  \T^d \to \R^{1,d}, \qquad
  (x,y) \mapsto (x_1 + \cdots + x_d +m, y_1, \ldots, y_d).
\]
This map is ``real''.  Since \(S_i\) and~\(S_i^*\) have the Fourier
transforms \(x_i \pm \ima y_i\), the Fourier transform of~\(H_m\) in
\(\Cont(\T^d,\Cliff_{1,d})\) is equal to
\(\tilde\beta_{1,d} \circ \tilde\varphi_m\).

\begin{lemma}
  \label{lem:phi_avoids_zero}
  If \(m \notin \{-d,-d+2,\ldots,d-2,d\}\), then
  \(\tilde\varphi_m(\T^d) \subseteq \R^{1,d} \setminus 0\).
\end{lemma}

\begin{proof}
  Assume \(\tilde{\varphi}_m(x,y) = 0\).  Then
  \(x_1+\dotsc+x_d + m =0\) and \(y_1 = y_2 = \cdots = y_d =0\).
  The latter forces \(x_i = \pm 1\) for \(i = 1,\ldots,d\), and then
  \(m = - \sum_{i=1}^d x_i\) must belong to
  \(\{-d,-d+2,\ldots,d-2,d\}\).
\end{proof}

From now on, we assume \(m \notin \{-d,-d+2,\ldots,d+2,d\}\).  By
Lemma~\ref{lem:phi_avoids_zero}, this is equivalent to~\(H_m\) being
invertible, which is needed to define its topological phase.

Under our assumption on~\(m\), there is a well defined ``real''
function
\[
  \varphi_m\colon \T^d \to \Sphere^{1,d}, \qquad
  \varphi_m(z) \defeq  \frac{1}{\norm{ \tilde{\varphi}_m(z)}} \cdot
  \tilde{\varphi}_m(z)
\]
as in~\eqref{eq:varphi_m}.  Since
\(\tilde\beta_{1,d}(x)^2 = \norm{x}^2\), the ``spectral flattening''
of~\(H_m\) with spectrum \(\{\pm1\}\) is the operator whose Fourier
transform in \(\Cont(\T^d,\Cliff_{1,d})\) is
\(\tilde\beta_{1,d} \circ \varphi_m\).  Since~\(\varphi_m\) takes
values in~\(\Sphere^{1,d}\) by construction, this is just
\(\beta_{1,d} \circ \varphi_m\) with the real, odd selfadjoint
unitary
\(\beta_{1,d} \in \Cont(\Sphere^{1,d}) \otimes \Cliff_{1,d}\)
in~\eqref{eq:beta}.  The composite
\(\beta_{1,d} \circ \varphi_m \in \mathcal{FU}(\Cont(\T^d) \otimes
\Cliff_{1,d})\) is its pullback along~\(\varphi_m\).

To get a class in van Daele's \(\K\)-theory, we must consider a
formal difference \([\beta_{1,d} \circ \varphi_m] - [f]\) for some
\(f\in \mathcal{FU}(\Cont(\T^d) \otimes \Cliff_{1,d})\).
Physically, \(f\) describes the topological phase that we choose to
call ``trivial''.  An obvious choice in our case is \(f=\gamma_0\),
the constant function on~\(\T^d\) with value~\(\gamma_0\).  Another
obvious choice would be~\(-\gamma_0\).  In the complex case, these
two are homotopic.  In the ``real'' case, however, these two choices
turn out to have different classes in
\(\KR\)\nb-theory for \(d \le 2\).
So the sign choice here actually matters.

\begin{lemma}
  \label{lem:Hm_as_composite}
  Up to the sign~\(\chi\) from
  Lemma~\textup{\ref{lem:beta_from_point}}, the class
  \([\beta_{1,d} \circ \varphi_m] - [\gamma_0]\) in
  \(\DK(\Cont(\T^d) \otimes \Cliff_{1,d})) \cong \KR^{-d}(\T^d)\) is
  the composite of the \(\KR\)\nb-oriented correspondences
  \[
    S_! \defeq \left(
      \begin{tikzcd}[sep=small]
        & (\pt,1) \ar[dl, equal] \ar[dr, "S"]\\
        \pt && \Sphere^{1,d}
      \end{tikzcd}
    \right),\qquad
    \varphi_m^* \defeq
    \left(
      \begin{tikzcd}[sep=small]
        & (\T^d,1) \ar[dl, "\varphi_m"'] \ar[dr, equal]\\
        \Sphere^{1,d} && \T^d
      \end{tikzcd}
    \right).
  \]
\end{lemma}

\begin{proof}
  The correspondence~\(S_!\) represents the class
  \(\chi\cdot ([\beta_{1,d}] - [\gamma_0])\) by
  Lemma~\ref{lem:beta_from_point}.  Here~\(\gamma_0\) denotes the
  constant function on~\(\Sphere^{1,d}\) with value~\(\gamma_0\), and
  we changed our numbering of Clifford generators to start at~\(0\).
  Composing with the correspondence denoted~\(\varphi_m^*\) pulls
  this back along the map~\(\varphi_m\).  This gives
  \(\chi\cdot ([\beta_{1,d} \circ \varphi_m] - [\gamma_0])\), where
  now~\(\gamma_0\) denotes the constant function on~\(\T^d\) with
  value~\(\gamma_0\).
\end{proof}

\begin{lemma}
  \label{lem:transverse_example}
  The two correspondences in
  Lemma~\textup{\ref{lem:Hm_as_composite}} are transverse.  So their
  composite is their intersection product.  The fibre product
  \(\T^d \times_{\Sphere^{1,d}} \pt\) in the intersection product is
  diffeomorphic to the finite subset
  \(\varphi_m^{-1}(S) \subseteq \T^d\).  For
  \(z\in\varphi_m^{-1}(S)\), let \(\sign(z)\) be the number
  of~\(-1\) among the coordinates of~\(z\).  Then
  \[
    [\beta_{1,d} \circ \varphi_m] - [\gamma_0]
    = \chi \sum_{z\in \varphi_m^{-1}(S)} (-1)^{\sign(z)} z_!
  \]
  holds in
  \(\DK(\Cont(\T^d)\otimes \Cliff_{1,d}) \cong \KR^{-d}(\T^d)\).
\end{lemma}

\begin{proof}
  We must show that the differential of~\(\varphi_m\) is a
  surjective map onto \(T_S \Sphere^{1,d}\) at all
  points~\((x,y)\in\T^d\) with \(\varphi_m(x,y) = S\).  The tangent
  space \(T_S \Sphere^{1,d}\) is the subspace \(\{0\} \times \R^d\),
  spanned by the basis vectors \(e_1,\dotsc,e_d\).  If
  \(\varphi_m(x,y) = S\), then \(y_1=y_2=\dotsb=y_d=0\) follows as
  in the proof of Lemma~\ref{lem:phi_avoids_zero}.  At these points,
  the tangent space of~\(\T^d\) is spanned by the vectors in the
  directions \(y_1,\dotsc,y_d\).  The differential
  of~\(\tilde\varphi_m\) maps these to the vectors
  \(e_1,\dotsc,e_d\).  On the preimage of~\(S\), the differential of
  the radial projection map
  \(\R^{1,d} \setminus \{0\} \to \Sphere^{1,d}\),
  \(z\mapsto z/\norm{z}\), just multiplies~\(e_j\) for
  \(j=1,\dotsc,d\) with a positive constant, so that the images
  still span \(T_S \Sphere^{1,d}\).  Thus
  \(\varphi_m\colon \T^d \to \Sphere^{1,d}\) is transverse to
  \(S\colon \pt \to \Sphere^{1,d}\).

  The canonical map
  \(\pi_{\T^d}\colon \T^d \times_{\Sphere^{1,d}} \pt \to \T^d\) is a
  diffeomorphism onto the closed submanifold
  \(\varphi_m^{-1}(S) \subseteq \T^d\) by the definition of the
  fibre product.  Since the differential of~\(\varphi_m\) is
  bijective at all points in the preimage of~\(S\), this preimage
  is discrete.  Since~\(\T^d\) is compact, it must be finite.  In
  fact, we may compute it easily: it consists of all points
  \((x,0) = (x_1,\dotsc,x_d,0,\dotsc,0) \in\T^d\) with
  \(\sum_{j=1}^d x_j + m <0\).  Here \((x,0)\in\T^d\) if and only if
  \(x_j \in \{\pm1\}\) for \(j=1,\dotsc,d\).

  Let \(z=(x,y)\in \T^d\) satisfy \(\rinv(z) = z\).  Then \(y=0\)
  and hence \(x_i \in \{\pm1\}\) for \(i=1,\dotsc,d\).  These points
  satisfy \(\varphi_m(z) \in \{N,S\}\) for the north and south pole
  in~\eqref{eq:poles}, simply because~\(\varphi_m\) is ``real'' and
  \(N,S\) are the only points fixed by the involution
  on~\(\Sphere^{1,d}\).  Give~\(\T^d\) the \(\KR\)\nb-orientation
  described in Example~\ref{exa:KR-oriented_manifolds}.  This
  induces a \(\KR\)\nb-orientation on the fibre \(T_z \T^d\).  Since
  the vector field generated by
  the exponential function points upwards at \((1,0) \in \T^1\) and
  downwards at \((-1,0)\in\T^1\), the projection to the
  \(y\)\nb-coordinate \(T_{(1,0)} \T^1 \to \R^{0,1}\) preserves the
  orientation at~\(+1\) and reverses it at~\(-1\).  Therefore, the
  projection to the \(y\)\nb-coordinate \(T_z \T^d \to \R^{0,d}\)
  multiplies the orientation with the sign~\((-1)^{\sign(z)}\).
  The sum in geometric bivariant \(\KR\)-theory is the disjoint
  union of correspondences.  Therefore, the discrete set
  \(\varphi_m^{-1}(S)\) in the composite correspondence contributes
  the sum of \((-1)^{\sign(z)} z_!\) over all
  \(z\in \varphi_m^{-1}(S)\).  Lemma~\ref{lem:Hm_as_composite}
  identifies this sum with
  \(\chi\cdot \bigl([\beta_{1,d} \circ \varphi_m] -
  [\gamma_0]\bigr)\).
\end{proof}

If \(d-2<m<d\), then \(\varphi_m^{-1}(S)\) has only one element,
namely, the single point \((-1,-1,\dotsc,-1,0,\dotsc,0) \in \T^d\)
with sign~\((-1)^d\).  Thus we get 
\begin{equation}
  \label{eq:m_highest_interval}
  [\beta_{1,d} \circ \varphi_m] - [\gamma_0]
  = \chi\cdot (-1)^d \cdot(-1,\dotsc,-1,0,\dotsc,0)_!.
\end{equation}
Up to a sign, this is the generator of \(\KR^{-d}(\T^d)\) that is
mapped to a generator of the \(\KR\)\nb-theory of the Roe
\(\Cst\)\nb-algebra of~\(\Z^d\)
(see~\cite{Ewert-Meyer:Coarse_insulators}).  It is argued
in~\cite{Ewert-Meyer:Coarse_insulators} why this generator describes
strong topological phases.  For other values of~\(m\), we would like
to simplify the formula in Lemma~\ref{lem:transverse_example}
further by comparing the point inclusions \(z_!\colon \pt \to \T^d\)
for different \(z\in\T^d\) with \(\rinv(z)=z\).  If we work in
complex \(\K\)\nb-theory, then all point inclusions are homotopic
and therefore give equivalent correspondences.  This fails, however,
in the ``real'' case.  We need some preparation to explain this.

\begin{lemma}
  \label{lem:Hm_zero_outside_range}
  If \(m < -d\), then~\(H_m\) is homotopic to~\(-\gamma_0\) in
  \(\mathcal{FU}(\Cont(\T^d) \otimes \Cliff_{1,d})\).
\end{lemma}

\begin{proof}
  The Hamiltonians~\(H_s\) for \(s\in (-\infty, m]\) give a
  homotopy of real, odd, selfadjoint unitaries \(\abs{H_s}^{-1} H_s\)
  between \(\abs{H_m}^{-1}H_m\) and
  \[
    \lim_{s\to-\infty} \abs{H_s}^{-1} H_s = -\gamma_0.\qedhere
  \]
\end{proof}  

The same argument shows that~\(H_m\) is homotopic to~\(\gamma_0\)
for \(m > d\).  This is consistent with
Lemma~\ref{lem:transverse_example}, which says that
\([H_m] - [\gamma_0] = 0\) in \(\KR^{-d}(\T^d)\) for \(m>d\).

Combining Lemmas \ref{lem:transverse_example}
and~\ref{lem:Hm_zero_outside_range} gives
\[
  [-\gamma_0] - [\gamma_0]
  = \chi \sum_{z\in \T^d,\ \rinv(z) = z} (-1)^{\sign(z)} z_!
\]
because if \(m<-d\), then~\(\varphi_m\) maps all points in~\(\T^d\)
with \(\rinv(z) = z\) to~\(S\).  In particular, for \(d=1\), this
says that
\begin{equation}
  \label{eq:points_in_T1}
  (1,0)_! =
  (-1,0)_! + \chi\cdot \bigl([-\gamma_0] - [\gamma_0]\bigr).
\end{equation}
The difference \([-\gamma_0] - [\gamma_0]\) is represented by
constant functions and thus is in the image of
\(\DK(\Cliff_{1,1}) \cong \KO^{-1}(\pt) \cong \Z/2\) in
\(\KR^{-1}(\T^1)\).  Since this group is \(2\)\nb-torsion, we may
drop the sign~\(\chi\) in~\eqref{eq:points_in_T1}.

\begin{lemma}
  The isomorphism above maps \([-\gamma_0] - [\gamma_0]\) to the
  nontrivial element \(\mu\in \KO^{-1}(\pt) \cong \Z/2\).
\end{lemma}

\begin{proof}
  Identify \(\Mat_n(\Cliff_{1,1}) \cong \hat\Mat_{2n}\).  The odd
  selfadjoint real unitaries in \(\Mat_n(\Cliff_{1,1})\) are
  identified with the \(2n\times 2n\)-matrices
  \[
    \begin{pmatrix}
      0&U\\U^*&0
    \end{pmatrix}
  \]
  for an orthogonal \(n\times n\)-matrix~\(U\).  Two such matrices
  are homotopic among odd selfadjoint unitaries if and only if the
  orthogonal matrices are in the same connected component of the
  orthogonal group or, equivalently, have the same determinant.
  Here the elements~\(\pm\gamma_0\) correspond to
  \(\pm1 \in \mathrm{O}(1) = \{\pm1\}\), which lie in the two
  different connected components.
\end{proof}

We recall how to describe the \(\KR\)\nb-theory of tori.  There is a split
extension of ``real'' \(\Cst\)\nb-algebras
\[
  \begin{tikzcd}
    \Cont_0(\R^{0,1}) \ar[r, >->] &
    \Cont(\T^1) \ar[r, ->>] &
    \C. \ar[l, bend right]
  \end{tikzcd}
\]
It induces a \(\KKR\)-equivalence
\(\Cont(\T^1) \cong \Cont_0(\R^{0,1}) \oplus \C\).  Since
\(\Cont(\T^d)\) is the tensor product of \(d\)~copies of
\(\Cont(\T^1)\) and the tensor product bifunctor is additive on
\(\KKR\), it follows that \(\Cont(\T^d)\) is \(\KKR\)-equivalent to
a direct sum of tensor products \(A_1 \otimes \dotsb \otimes A_d\),
where each~\(A_j\) is either \(\Cont_0(\R^{0,1})\) or~\(\C\).  We
label such a summand by the set \(I\subseteq \{1,\dotsc,d\}\) of
those factors that are~\(\C\).

\begin{theorem}
  \label{the:compute}
  Let \(m \in (-d+2n, -d+2n+2)\) for some \(n\in\{0,\dotsc,d-1\}\).
  Let~\(\chi\) be the sign from
  Lemma~\textup{\ref{lem:beta_from_point}}.  The image of
  \([H_m] - [\gamma_0] \in \KR^{-d}(\T^d)\) in the summand
  \(\KR^{-d}(\R^{0,d-\abs{I}})\) corresponding to
  \(I\subseteq \{1,\dotsc,d\}\) is computed as follows:
  \begin{itemize}
  \item if \(I=\emptyset\), the image in
    \(\KR^{-d}(\R^{0,d}) \cong \Z\) is
    \((-1)^d\chi \cdot \binom{d-1}{n}\);
  \item if \(\abs{I}=1\) and \(n\ge1\), the image in
    \(\KR^{-d}(\R^{0,d-1}) \cong \Z/2\) is
    \(\binom{d-2}{n-1} \bmod 2\);
  \item if \(\abs{I}=2\) and \(n\ge2\), the image in
    \(\KR^{-d}(\R^{0,d-2}) \cong \Z/2\) is
    \(\binom{d-3}{n-2} \bmod 2\);
  \item it is zero in the other cases, that is, for \(\abs{I}\ge3\)
    or \(n < \abs{I}\).
  \end{itemize}
\end{theorem}

\begin{proof}
  We first consider the image of \((x,0)_!\) in the summand labeled
  by~\(I\).  The shriek map \((x_1,\dotsc,x_d,0,\dotsc,0)_!\) is the
  exterior product of \((x_i,0)_!\) for \(i=1,\dotsc,d\) (see
  Remark~\ref{rem:exterior}).  In the decomposition of
  \(\KR^{-1}(\T^1) \cong \KR^{-1}(\R^{0,1}) \oplus \KO^{-1}(\pt)
  \cong \Z \oplus \Z/2\), \((-1,0)_!\) becomes \((\pt_!,0)\) with
  the standard generator \(\pt_!\) of
  \(\KR^{-1}(\R^{0,1}) \cong \Z\).  Equation~\eqref{eq:points_in_T1}
  shows that \((1,0)_! = (\pt_!,\mu)\) with the nontrivial element
  \(\mu\in \KO^{-1}(\pt) \cong \Z/2\).  So \((x,0)_!\) is the
  exterior product of \(d\)~factors that are \((\pt_!,0)\) if
  \(x_i=-1\) and \((\pt_!,\mu)\) if \(x_i = 1\).  The component in
  the summand labeled by~\(I\) is zero unless \(x_i = +1\) for all
  \(i\in I\).  If \(x_i = +1\) for all \(i\in I\), then we get the
  exterior product of \(\pt_!  \in \KR^{-1}(\R^{0,1})\) for all
  \(i\notin I\) and \(\mu\in \KO^{-1}(\pt)\) for all \(i\in I\).
  The exterior product
  \(\mu\otimes \mu \in \KO^{-2}(\pt \times \pt) \cong \Z/2\) is
  known to be the nontrivial element (this also follows from the
  discussion above Proposition~\ref{pro:other_sphere_generators}),
  whereas
  \(\mu\otimes \mu \otimes \mu\) and hence also all higher exterior
  products of~\(\mu\) vanish because \(\KO^{-3}(\pt) = 0\).  Thus
  the image of \((x,0)_!\) is zero for all summands with
  \(\abs{I} \ge 3\) or \(x_i =-1\) for some \(i\in I\), and the
  standard generator of \(\KR^{-d}(\R^{0,d-\abs{I}})\) if
  \(\abs{I} \le 2\) and \(x_i = +1\) for all \(i\in I\).

  Now we sum up \((-1)^{\sign(x,0)} (x,0)_!\) over all
  \((x,0)\in \T^d\) with \(\varphi_m(x,0) = S\) or, equivalently,
  \(\sum_{i=1}^d x_i + m <0\).  The latter means that the number
  of~\(+1\) among the coordinates~\(x_i\) is at most~\(n\).  There
  are \(\binom{d}{j}\) points \((x,0)\) with \(x_i=+1\) for
  exactly~\(j\) indices~\(i\).  For all of them, the image of
  \((x,0)_!\) in the direct summand \(\KR^{-d}(\R^{0,d}) \cong \Z\)
  for \(I=\emptyset\) is the same standard generator.  Hence the
  image of \([H_m] - [\gamma_0]\) in this direct summand is
  \[
    (-1)^d\chi \cdot \sum_{j=0}^n (-1)^{n-j} \binom{d}{j}
    = (-1)^d\chi \cdot\binom{d-1}{n}
  \]
  because of~\eqref{eq:m_highest_interval}.
  Now
  let \(\abs{I}=1\), so that \(I = \{i_0\}\) for some
  \(i_0\in \{1,\dotsc,d\}\).  The corresponding direct summand
  \(\KR^{-d}(\R^{0,d-1}) \cong \Z/2\) only sees
  \((x,0)\in \varphi_m^{-1}(S)\) with \(x_{i_0} = 1\).  There are
  \(\binom{d-1}{j-1}\) points with \(x_{i_0} = 1\) and \(x_i=+1\)
  for exactly~\(j\) indices \(i=1,\dotsc,d\).  Thus the overall
  contribution in this summand isomorphic to~\(\Z/2\) vanishes if
  \(n=0\) and otherwise is equal to the class mod~\(2\) of
  \[
    \sum_{j=1}^n
    (-1)^{n-j} \binom{d-1}{j-1}
    = - \sum_{j=0}^{n-1}
    (-1)^{n-j} \binom{d-1}{j}
    = \binom{d-2}{n-1};
  \]
  we could leave out all signs because \(+1 = -1\) in~\(\Z/2\).
  Similarly, for \(I \subseteq \{1,\dotsc,d\}\) with \(\abs{I}=2\),
  we get \(\binom{d-3}{n-2}\) if \(n\ge2\)
  and~\(0\) if \(n\le 1\).
\end{proof}

The formula in Theorem \ref{the:compute} is compatible with the
Chern character computation in
\cite{Prodan-Schulz-Baldes:Bulk_boundary}*{Equation~(2.26)}.  The
latter, however, only gives partial information about the
\(\K\)\nb-theory class, even in the complex case, because it only
concerns the top-dimensional part of the Chern character.

As a result, we find that the Hamiltonian~\(H_m\) for \(d-2< m<d\)
represents a generator of the direct summand
\(\KR^{-d}(\R^{0,d}) \cong \Z\) in \(\KR^{-d}(\T^d)\).  We get
different \(\KR\)-classes by stacking insulators in lower dimension
along some direction.  In our framework, this means that we consider
a ``coordinate'' projection \(\varrho\colon \T^d \prto \T^k\) for
some \(0\le k\le d\) and some
\(1 \le i_1 < i_2 < \dotsb < i_k \le n\), which only keeps the
coordinates \(x_{i_j},y_{i_j}\) for \(j=1,\dotsc,k\).  Then we may
pull back the generator of
\(\KR^{-k}(\R^{0,k}) \subseteq \KR^{-k}(\T^k)\) of the form~\(H_m\)
along the map~\(\varrho\) to a class in \(\KR^{-k}(\T^d)\); this is
given by the Hamiltonian
\begin{equation}
  \label{eq:Hm_on_Tk}
  H_m \defeq \frac{1}{2\ima} \sum_{j=1}^k (S_{i_j} - S_{i_j}^*) \otimes \gamma_j
  + \Bigl(m + \frac{1}{2} \sum_{j=1}^k (S_{i_j} + S_{i_j}^*)\Bigr)
  \otimes \gamma_0
  \in\Cst(\Z^d)\otimes\Cliff_{1,k}
\end{equation}
for \(m\in (k-2,k)\).  This covers all the summands
\(\KR^{-k}(\R^{0,k})\) in \(\KR^*(\T^d)\).
Proposition~\ref{pro:other_sphere_generators} gives explicit
generators for the summands isomorphic to \(\Z/2\) as well: simply
view~\(H_m\) in~\eqref{eq:Hm_on_Tk} as taking values in
\(\Cliff_{1,k+d}\) for \(d=1,2\).  We do not describe generators for
the remaining summands of the form
\(\KR^{-k+4}(\R^{0,k}) \cong \Z\).

\section{Transfer to Hilbert space}
\label{sec:to_Hilbert}

Now we transfer the Clifford
algebra-valued function~\(H_m\) to an operator on a Hilbert space
with extra symmetries depending on the Clifford algebra (see
also~\cite{Kellendonk:Cstar_phases}).  We represent \(\Cont(\T^d)\)
on \(\ell^2(\Z^d)\) in the usual way.  Then
\(\Cont(\T^d) \otimes \Cliff_{a,b}\) is represented on
\(\ell^2(\Z^d)\otimes \C^k\) for some \(k\in\N\), with some extra
symmetries acting on~\(\C^k\).  Here \(k\in\N\) and the symmetry
type depend on \(a,b\in\N\), and so there is a number of cases to
consider.  We let \(j\defeq b-a+1\bmod 8\), so that
\(\DK(\Cont(\T^d) \otimes \Cliff_{a,b}) = \KR^{-j}(\T^d)\).  The
symmetry type depends only on \(j\in\Z/8\) because Clifford algebras
with the same~\(j\) are Morita equivalent as graded ``real''
algebras.  Therefore, it will suffice to look at one representative
Clifford algebra for each \(j\in\Z/8\).

First assume that \(j\) is even or, equivalently, \(b-a\) is odd.
Then there is an isomorphism
\(\Cliff_{a,b} \cong \Mat_{2^k}\C \oplus \Mat_{2^k}\C\) for
\(k= (a+b-1)/2\) such that the \(\Z/2\)-grading
automorphism~\(\alpha\) merely flips the two
summands~\(\Mat_{2^k}\C\).  Thus selfadjoint, odd elements
of~\(\Cliff_{a,b}\) become elements of the form
\((x,-x)\in \Mat_{2^k}\C \oplus \Mat_{2^k}\C\) for some
\(x\in\Mat_{2^k}\C\) with \(x=x^*\).  As a ring automorphism, the
``real'' involution~\(\rinv\) on \(\Cliff_{a,b}\) permutes the two
direct summands~\(\Mat_{2^k}\C\).  It induces either the trivial or
the nontrivial permutation.  We first assume that the permutation is
trivial.  Then~\(\rinv\) restricts to the same ``real'' involution
on both summands~\(\Mat_{2^k}\C\) because it commutes
with~\(\alpha\).  This involution is implemented as conjugation
by~\(\Theta\) for an antiunitary operator
\(\Theta\colon \C^{2^k} \to \C^{2^k}\).  The pair \((x,-x)\) for
\(x\in\Mat_{2^k}\C\) is fixed by~\(\rinv\) if and only if~\(x\)
commutes with~\(\Theta\).  Thus we identify ``real'' selfadjoint odd
unitaries in~\(\Cliff_{a,b}\) with selfadjoint unitaries
in~\(\Mat_{2^k}\C\) that commute with~\(\Theta\).  In other words,
we are dealing with systems with a time-reversal
symmetry~\(\Theta\).

Since~\(\Theta\) induces an antilinear involution,
\(\Theta^2 = \pm1\).  If \(\Theta^2 = +1\), then the real subalgebra
\((\Mat_{2^k}\C)_\R\) fixed by~\(\rinv\) is~\(\Mat_{2^k}\R\) and
\((\Cliff_{a,b})_\R \cong \Mat_{2^k}\R \oplus \Mat_{2^k}\R\); this
happens for \(\Cliff_{1,0}\) and thus for \(j \equiv 0 \bmod 8\).
If \(\Theta^2 = -1\), then \(k\ge1\) and
\((\Mat_{2^k}\C)_\R \cong \Mat_{2^{k-1}}\Qut\) and
\((\Cliff_{a,b})_\R \cong \Mat_{2^{k-1}}\Qut \oplus
\Mat_{2^{k-1}}\Qut\) for the quaternions~\(\Qut\); this happens for
\(\Cliff_{0,3}\) and thus for \(j \equiv 4 \bmod 8\).

Next, we assume that~\(\rinv\) flips the two
summands~\(\Mat_{2^k}\C\).  Then \(\rinv \circ \alpha\) maps each
direct summand into itself and induces the same real involution on
both summands~\(\Mat_{2^k}\C\).  Thus we may implement
\(\rinv\circ\alpha\) by an antiunitary operator
\(\Theta\colon \C^{2^k} \to \C^{2^k}\) as above.  The difference is
that a pair \((x,-x)\) is real if and only if \(\Theta\)
\emph{anti}commutes with~\(x\).  Thus~\(\Theta\) is now a
particle-hole symmetry.

Since~\(\rinv\) flips the two summands, the real subalgebra
\((\Cliff_{a,b})_\R\) is isomorphic to~\(\Mat_{2^k}\C\), identified
with the subalgebra of
\((x,\Theta^{-1} x \Theta) \in \Mat_{2^k}\C \oplus \Mat_{2^k}\C\)
for \(x\in \Mat_{2^k}\C\).  Once again, there are the two
possibilities \(\Theta^2=\pm1\).  If \(\Theta^2=+1\), then the even
subalgebra of \((\Cliff_{a,b})_\R\) is \(\Mat_{2^k}\R\); this
happens for \(\Cliff_{0,1}\) and thus for \(j \equiv 2 \bmod 8\).
If \(\Theta^2=-1\), then \(k\ge1\) and the even subalgebra of
\((\Cliff_{a,b})_\R\) is \(\Mat_{2^{k-1}}\Qut\); this happens for
\(\Cliff_{3,0}\) and thus for \(j \equiv 6 \bmod 8\).

Now assume that \(a-b\) is even.  Then
\(\Cliff_{a,b} \cong \Mat_{2^k}\C\) for some \(k\in\N\).  The
\(\Z/2\)-grading and the ``real'' involution are implemented by a
unitary operator \(\Xi\colon \C^k \to \C^k\) and an antiunitary
operator \(\Theta\colon \C^k \to \C^k\).  A real, selfadjoint odd
unitary in~\(\Cliff_{a,b}\) then becomes a selfadjoint unitary
in~\(\Mat_{2^k}\C\) that anticommutes with~\(\Xi\) and commutes
with~\(\Theta\).  Thus it has~\(\Theta\) as a time-reversal
and~\(\Xi\) as a chiral symmetry.  Then \(\Theta \Xi\) is a
particle-hole symmetry.  Multiplying~\(\Xi\) with a scalar, we may
arrange that \(\Xi^2 = 1\).  Since~\(\Theta\) induces an antiunitary
involution, \(\Theta^2 = \pm1\).  Since~\(\rinv\) commutes with the
grading, \(\Xi^{-1} \Theta \Xi = \pm\Theta\).  This is equivalent to
\((\Xi \Theta)^2 = \pm \Theta^2\).  So there are four possibilities
for the signs.  If \(\Theta^2 = +1\), then
\((\Cliff_{a,b})_\R\cong\Mat_{2^k}\R\).  If \(\Theta^2 = -1\), then
\(k\ge1\) and \((\Cliff_{a,b})_\R \cong\Mat_{2^{k-1}}\Qut\).  The
even subalgebra of~\(\Cliff_{a,b}\) is
\(\Mat_{2^{k-1}}\C \oplus \Mat_{2^{k-1}}\C\) if \(k\ge1\).  If
\(\Xi^{-1} \Theta \Xi = \Theta\), then the real involution
restricted to the even part preserves the two direct
summands~\(\Mat_{2^{k-1}}\C\), so that the even real subalgebra of
\(\Cliff_{a,b}\) is a direct sum of two simple algebras.  If
\(\Xi^{-1} \Theta \Xi = -\Theta\), however, then the real involution
restricted to the even part flips the two direct summands, so that
the even real subalgebra of \(\Cliff_{a,b}\) is simple.  An
inspection now shows the following:
\begin{itemize}
\item \(\Theta^2 = +1\), \(\Xi^{-1} \Theta \Xi = +\Theta\), and
  \((\Xi \Theta)^2 = +1\) for \(\Cliff_{1,1}\) with
  \(j\equiv 1 \bmod 8\);
\item \(\Theta^2 = +1\), \(\Xi^{-1} \Theta \Xi = -\Theta\), and
  \((\Xi \Theta)^2 = -1\) for \(\Cliff_{2,0}\) with
  \(j\equiv 7 \bmod 8\);
\item \(\Theta^2 = -1\), \(\Xi^{-1} \Theta \Xi = -\Theta\), and
  \((\Xi \Theta)^2 = +1\) for \(\Cliff_{0,2}\) with
  \(j\equiv 3 \bmod 8\);
\item \(\Theta^2 = -1\), \(\Xi^{-1} \Theta \Xi = +\Theta\), and
  \((\Xi \Theta)^2 = -1\) for \(\Cliff_{0,4}\) with
  \(j\equiv 5 \bmod 8\).
\end{itemize}
The correspondence between~\(j\) and the symmetry types is the same
as in \cite{Schulz-Baldes:Insulators}*{Table~1}.

\end{document}